\documentclass[11pt,reqno]{amsart}
\usepackage{amssymb}
\usepackage{amsmath,amsthm}
\usepackage{amsfonts}
\usepackage{leftidx}
\usepackage{mathrsfs}
\usepackage{enumerate}
\usepackage{abstract}
\usepackage{stmaryrd}

\def\R{\mathbb{R}^{3}}
\def\C{\mathbb{C}^{4}}

\def\p{\partial}

\def\R{\mathbb{R}^{3}}

\newtheorem{theorem}{Theorem}[section]
\newtheorem{lemma}[theorem]{Lemma}

\theoremstyle{definition}

\theoremstyle{remark}
\newtheorem{remark}[theorem]{Remark}

\numberwithin{equation}{section}

\begin{document}

\title[3D DKG in the scalar invariant space]{On Global Existence of 3D Charge Critical Dirac-Klein-Gordon system}
\author{Xuecheng Wang}
\address{Mathematics Department, Princeton University, Princeton, New Jersey,08544, USA}
\email{xuecheng@math.princeton.edu}
\thanks{}
\maketitle
\begin{abstract}
In this paper, we prove the global well-posedness property of charge critical Dirac-Klein-Gordon
 (DKG) system in $\mathbb{R}^{3+1}$ for small initial data in  a space of scale invariant data which has extra weighted regularity in the angular variables. Therefore, 
by finite speed propagation, 
we could also derive the local well-posedness property for large initial
 data in the same space. To author's knowledge, our result appears to be the first result 
on the critical 3D DKG. 
\end{abstract}
\tableofcontents

\section{Introduction}\label{introduction}
This paper is devote to study the global well-posedness(GWP) of the coupled
Dirac-Klein Gordon (DKG) system for small initial data at the critical level of regularity in 
the three spatial dimension setting,
and from the finite speed propagation property of 
wave equations and  Klein-Gordon equations, we can derive the local 
well-posedness (LWP) for large initial data as a byproduct. In general, the coupled Dirac-Klein-Gordon system has the following formulation:

\begin{equation}\label{DKG}
\mbox{(DKG)}\,\,\left\{\begin{array}{ll}
(-i\gamma^{\mu} \partial_{\mu} + M)\psi = \phi \psi& M \geq 0\\
(-\square + m^2) \phi = \psi^{\dagger} \gamma^0 \psi& \square = -\partial_t^2 + \Delta, m \geq 0
\end{array}
\right.
\end{equation}
where $\phi :\mathbb{R}^{1+n} \rightarrow \mathbb{R}$ represents a meson field and $\psi: \mathbb{R}^{1+n} \rightarrow \mathbb{C}^N$ is the Dirac
spinor field and could be viewed as a vector field in $\mathbb{C}^N$. 
The dimension of the spin space  $N$ depends on the
spatial dimension $n$. We use coordinates $t= x^0, x= (x^1, x^2,\cdots, x^n)$ on $\mathbb{R}^{1+3}$ and 
$\p_{\mu}$ denotes $\p /\p_{x^\mu}$. We will use the convention that the Greek indices 
range from $0$ to $n$, Roman indices range from $1$ to $n$ and the Einstein summation rule
is applied throughout this paper. 
 For each $\mu \in \{0, 1, \cdots, n\}$, $\gamma^{\mu}$  in (\ref{DKG}) represents a
 $N\times N$ Dirac Matrices, and the following rules are satisfied:
\begin{equation}\label{rules}
\gamma^{\mu}\gamma^{\nu} + \gamma^{\nu}\gamma^{\mu} = 2 g^{\mu\nu} I,\qquad (\gamma^{0})^{\dagger} = \gamma^{0}\quad and\,\, (\gamma^{j})^{\dagger}= - \gamma^{j},
\end{equation}
where $g^{\mu\nu} = (diag(1,-1,\cdots,-1))_{\mu,\nu}$, $I$ is 
the $N\times N$ identity matrix and the operator ``$\dagger$'' denotes the 
conjugate transpose. For each spatial 
dimension $n$, we are interested in the smallest possible 
dimension $N$ of the spin space
 that admits Dirac matrices satisfy (\ref{rules}). In particular, when $n=3$, the smallest
possible value is $4$, i.e, $N=4$ and we would like to mention that when $n=1,2$, $N=2$. 
In this paper, we will focus on the 3D case, i.e $n=3$. In 3D case, the Dirac matrices are given by 
the following: 
\begin{equation}
\gamma^0 = \left(\begin{array}{lr}
I & 0\\
0 & -I\\
\end{array}\right),\qquad \gamma^j = \left(\begin{array}{lr}
0 & \sigma^j \\
-\sigma^j & 0\\
\end{array}\right) ,
\end{equation}
where
\begin{equation}
\sigma^1=\left(\begin{array}{lr}
0 & 1\\
1 &0\\
\end{array}\right),\quad \sigma^2 = \left(\begin{array}{lr}
0 & -i\\
i & 0\\
\end{array}\right),\quad \sigma^3 = \left(\begin{array}{lr}
1 & 0\\
0 & -1\\
\end{array}\right)
\end{equation}
are Pauli matrices. Denote
\begin{displaymath}
\beta = \gamma^0,\quad \alpha^j = \gamma^0\gamma^j,
\end{displaymath}
 hence from (\ref{rules}), we could derive that $\beta^2 = (\alpha^j)^2 = I$, $\beta^{\dagger} = \beta$ and $(\alpha^j)^{\dagger} = \alpha^j$. 
With above notations, we can very easily reduce (\ref{DKG}) into the following formulation:
\begin{equation}\label{DKG5}
\left\{\begin{array}{lr}
-i (\partial_{t} + \alpha \cdot \nabla) \psi + M \beta \psi=   \phi \beta \psi, &\, \\
(-\square + m^{2}) \phi = \langle \beta \psi, \psi \rangle_{\C}.&\, \\
\end{array}\right.
\end{equation}
From now on, we will mainly working on the system (\ref{DKG5}). One can verify that the Dirac-Klein-Gordon system has
 two conservation laws : energy conservation law and charge conservation law. More precisely,
\begin{equation}
\int |\psi(t, x )|^2 d x = \int |\psi(0, x)|^2 d x, \quad \int e(\phi, \psi)(t, x) d x = \int e(\phi, \psi)(0, x) d x,
\end{equation}
where $e(\phi,\psi)$ is the energy density and has the following form:
\begin{equation}\label{equation11}
e(\phi, \psi) = Im(\psi^{\dagger}\alpha^{j} \partial_j \psi) - (M - g\phi) \psi^{\dagger} \beta \psi - 1/2((\partial_t \phi)^2 + |\nabla \phi|^2 + m^2 \phi^2).
\end{equation}
One might notice that the energy density $e(\phi, \psi)$ is not positive definite, which
is unpleasant and it's very difficult to exploit the energy conservation law to control
the solution. While, the charge conservation law is indeed 
very helpful, as it states that the $L^2$ norm of the spinor field $\psi$ doesn't change
with respect to time. This is also one of the key ingredients in Chadam's  \cite{Chadam1} proof of
global regularity for the 1D DKG system.

As the goal of this paper is to study global well-posedness in the critical 
regularity space. It would be better to have an idea of what the critical regularity is
at first. To this end, we apply the scaling heuristic,  in the massless 
case $m=M=0$,  DKG system (\ref{DKG}) is invariant under the scaling:
\begin{displaymath}
\psi(t, x) \rightarrow \frac{1}{L^{3/2}} \psi(\frac{t}{L}, \frac{x}{L}),\qquad \phi (t, x) \rightarrow \frac{1}{L} \phi (\frac{t}{L}, \frac{x}{L}),
\end{displaymath}
hence the scale invariant initial data space in 
3D is
\begin{equation}
(\psi_{0}, \phi_{0}, \phi_{1}) \in L^{2}(\R) \times \dot{H}^{1/2}(\R)\times \dot{H}^{-1/2}(\R),
\end{equation}
heuristically, one doesn't expect well-posedness below this critical regularity.

\subsection{Previous results}\label{previous}
\par By classical energy method, one can obtain LWP of 3D DKG system for initial data lies in  $H^{1+\epsilon}\times H^{3/2+\epsilon}\times H^{1/2+\epsilon}$ for any 
$\epsilon > 0$. Later, Bachelot \cite{Bachelot} proved that
 this $\epsilon$ could be removed. By using Strichartz estimates for the homogeneous wave 
equations, one can lower the regularity to $H^{1/2 + \epsilon} \times H^{1+\epsilon} \times H^{\epsilon}$ (see
\cite{Bournaveas1,Ponce}). 
To lower the regularity further, we need to utilize the
special structures inside the nonlinearities.

\vspace{1\baselineskip}

\par One may notice that the nonlinearities of the DKG system are quadratic and due to the well-known counterexample found by 
Lindblad \cite{Lindblad}, we know that the level of regularity predicted by scaling cannot be reached for general quadratic 
nonlinear wave equations. The main enemy  of preventing to reach the critical regularity is the strong parallel interaction inside the quadratic nonlinearity 
and such interaction is especially strong in the low spatial dimensions, like $n=1,2,3$. 
But such parallel interactions might be eliminated if null structure presents inside the nonlinearity, which makes it possible to reach the critical regularity.
\vspace{1\baselineskip}
\par In \cite{Klainerman}, Klainerman and Machedon demonstrated a null structure in the nonlinearity of Klein-Gordon part, and later Bournaveas\cite{Bournaveas} followed their idea and found
a null structure in the nonlinearity of Dirac part of DKG system and applied it successfully to lower the regularity to $H^{1/2}\times H^{1}\times L^{2}$. However, in \cite{Bournaveas}, $X^{s,b}$ type Bourgain space
was not used to maximize the advantage of null structure. Recently, Fang and Grillakis \cite{Fang2}  proved LWP in $H^{s}\times H^{1}\times L^{2}$ for 
$1/4 < s \leq 1/2$ by using Bourgain spaces to utilize most of the null structures.

\vspace{1\baselineskip}

\par However the null structure found by Bournaveas has a drawback that it involves squaring the Dirac equation which causes difficulty at very low regularity and not as good as the null structure that appears in the Klein-Gordon part. Recently P. D'Ancona, D. Foschi, and S. Selberg  \cite{D'Ancona1} found that the null structure inside the Klein-Gordon part  also presents in the Dirac part.  After using duality argument, one can see that the type of null structure inside the Dirac part is of same type as appeared inside the meson field. Because of this observation which simplifies the analysis of the DKG system, they proved
the LWP at the regularity level arbitrarily close to the scale invariant level, i.e LWP holds for DKG system with initial data
 $(\psi_{0}, \phi_{0}, \phi_{1}) \in H^{\epsilon}\times H^{1/2+\epsilon}\times H^{-1/2+\epsilon}$, and  $\epsilon>0$ can be arbitrarily small.
 This $\epsilon$ gap is very helpful in the local theory of subcritical case, 
as one can  gain $\sigma(T)$ in the bilinear estimate for the nonlinearities,
 here $T$ stands for the length of time interval of existence, $\sigma(T) > 0$ depends
continousley on T and satisfies $\lim_{T\rightarrow 0^+} \sigma(T) = 0$. The gained
$\sigma(T)$ is sufficient to make
contraction argument works and get the LWP for arbitrary large
 initial data in $H^{\epsilon}\times H^{1/2+\epsilon}\times H^{-1/2+\epsilon}$.

\vspace{1\baselineskip}
\par Thus the remained interesting open problem is to 
understand what would happen in the critical regularity space. In the 
scale invariant space setting, we don't have the luxury of gain
$\sigma(T)$ in the bilinear estimate, as otherwise, we would loss derivative and can't close the argument to get LWP for large data. 

\subsection{Statement of the main result}
In this paper, we prove that at least for small initial data in a Besov type space which has extra regularity in the angular variables and at the critical regularity, the solution is global well-posed and 
scattering. Define the function spaces $\dot{B}^{r, s}_{2,1}$ and $\dot{H}^{s}_{\Omega}$  by the following norms:
\[
\| f \|_{\dot{B}^{r, s}_{2,1}} := \| \langle \Omega \rangle^{s} f \|_{\dot{B}_{2,1}^{r}}, \quad\| f \|_{\dot{H}^{s}_{\Omega}} = \| \langle \Omega \rangle f \|_{\dot{H}^{s}}, \quad \langle \Omega \rangle^{s} f := f_{0} + (- \Delta_{\mathbb{S}^{2}})^{s/2} f.
\]
Here $\dot{B}_{2,1}^{r}$ is the homogeneous Besov space and $f_{0}$ is the radial part of $f$, which is 
\[
f_{0}(r)  = \frac{1}{m(\mathbb{S}^{2})} \int_{\mathbb{S}^{2}} f(r \omega) d \omega.
\]
Our main result is the following:
\begin{theorem}\label{maintheorem}
 For the DKG system (\ref{DKG5}) in (1+3)-Minkowski space with ${2}M > m>0$ or $M=m=0$. There are
 exist constants $C_{1}$ and $C_{2}$, such that for small initial data $ (\psi_{0},\phi_{0},\phi_{1}) \in \dot{B}_{2,1}^{0,1}\times \dot{B}_{2,1}^{1/2,1}\times \dot{B}^{-1/2,1}_{2,1}$
satisfies
\begin{equation}\label{radial}
  \|  (\psi_{0},\phi_{0},\phi_{1}) \|_{\dot{B}_{2,1}^{0,1}\times \dot{B}_{2,1}^{1/2,1}\times \dot{B}^{-1/2,1}_{2,1}} \leq C_{1},
\end{equation}
then there exists a global solution $(\psi(t), \phi(t))$ to (\ref{DKG5}) such that 
\begin{displaymath}
 \sup_{t\in (-\infty, +\infty)} \|(\psi(t),\phi(t), \partial_t \phi(t)\|_{\dot{B}_{2,1}^{0,1}\times \dot{B}_{2,1}^{1/2,1}\times \dot{B}^{-1/2,1}_{2,1}} \leq C_{2}.
\end{displaymath}
Moreover, the  solution depends smoothly on the initial data. Furthermore, if the initial data has extra smoothness, then the solution retains this extra smoothness. More precisely, if $(\psi_{0},\phi_{0},\phi_{1})$ also has finite 
$\dot{H}^{s}_{\Omega}\times \dot{H}^{r}_{\Omega}\times \dot{H}^{r-1}_{\Omega}$ norm for $s > 0 $ and $r > 1/2$, then there exists constant $C_{3}$ such that
\begin{equation}
\sup_{t\in (-\infty, +\infty)} \| (\psi(t),\phi(t),\partial_{t}\phi(t)) \|_{\dot{H}^{s}_{\Omega}\times \dot{H}^{r}_{\Omega}\times \dot{H}^{r-1}_{\Omega}} \leq C_{3}.
\end{equation}
\end{theorem}
As a byproduct of the proof of the theorem \ref{maintheorem}, we can also
derive the asymptotic behavior of solution. More precisely, 

\begin{theorem}\label{scattering}
Given any sufficiently small initial data $ (\psi_{0},\phi_{0},\phi_{1}) \in \dot{B}_{2,1}^{0,1}\times \dot{B}_{2,1}^{1/2,1}\times \dot{B}^{-1/2,1}_{2,1}$
satisfies (\ref{radial}), then there exists unique functions $(\psi_{0}^{+},\phi_{0}^{+},\phi_{1}^{+})$
and $(\psi_{0}^{-},\phi_{0}^{-},\phi_{1}^{-})$ $\in\dot{B}_{2,1}^{0,1}\times \dot{B}_{2,1}^{1/2,1}\times \dot{B}^{-1/2,1}_{2,1}$ such that the solution $(\psi(t),\phi(t))$ of (\ref{DKG5}) with initial data $(\psi_{0},\phi_{0},\phi_{1})$ approaches to the the solution 
$(\psi^{\pm}(t),\phi^{\pm}(t))$ of the associated linear system of (\ref{DKG5}) with initial data $(\psi_{0}^{\pm},\phi_{0}^{\pm},\phi_{1}^{\pm})$ as $t\rightarrow \pm\infty$. More precisely 
\begin{equation}
 \lim_{t\rightarrow \pm \infty} \| \psi(t) - \psi^{\pm}(t) \|_{\dot{B}^{0,1}_{2,1}} + \| \phi(t) - \phi^{\pm} (t) \|_{\dot{B}^{1/2,1}_{2,1}} + \| \partial_{t} \phi(t) - \partial_{t} \phi^{\pm}(t) \|_{\dot{B}^{-1/2,1}_{2,1}} = 0.
\end{equation}
Furthermore, the scattering operator which maps $(\psi_{0},\phi_{0},\phi_{1}) $ to $(\psi_{0}^{\pm},\phi_{0}^{\pm},\phi_{1}^{\pm}) $ is a local diffeomorphism in $ \dot{B}_{2,1}^{0,1}\times \dot{B}_{2,1}^{1/2,1}\times \dot{B}^{-1/2,1}_{2,1}$.
\end{theorem}

\subsubsection{Overview of the main difficulties}\label{maindiff}

\par Unlike in the local theory, if one wants to prove the global well-posedness,
 we can't treat the linear mass terms appeared in the left hand side of (\ref{DKG5}) as a part of
nonlinearities any more, the type of the leading linear equation matters in the evolution of global solutions. 
 When $M > 0$ ($m > 0$), the Dirac spinor field $\psi$ (meson field $\phi$) of (\ref{DKG5}) is of Klein-Gordon type and
 when $M=0$ ($m=0$),
the Dirac spinor field $\psi$ (meson field $\phi$) is of wave type.   
One of the main difficulties comes from the ratio of the size of $M$ and $m$.  One might notice from the statement  of theorem 
\ref{maintheorem} that the cases: when ${2} M \leq m$ and when $m=0, M>0$ are excluded. We postpone the explanation of why those cases are excluded and first give an overview of 
the full picture. 

\vspace{1\baselineskip}
\par In the proof of the bilinear estimate, which is crucial to close argument,
 the most troublesome type of interaction is when both the inputs and the output are close their corresponding
 characteristic hypersurfaces. In order to control this type of interaction, certain cancellation is needed. While luckily, 
inside the DKG system (\ref{DKG5}), there exists null structures inside both Dirac spinor field part and  the meson
field part. However, there are slightly difference between the null structures inside those two parts. The null
structure inside meson field is \emph{explicit} inside the nonlinearity, while as mentioned in the subsection \ref{previous}, 
the same type null structure inside the Dirac part revealed by D'Ancona, Foschi and
Selberg \cite{D'Ancona1} via \emph{duality argument}. In both parts, the null structure can help us to gain one degree of angle between the spatial frequencies of two inputs. While there is a potential problem to use duality argument in our setting, as one has to put the duality element into the space that has \emph{negative regularity} in angular variables, which is not pleasant  and will cause additional problems in the bilinear estimate. However, as one can see from (\ref{equation20001}), without duality argument, we can still gain one degree of angle, but now the angle is between the frequencies of Dirac input and the output. There are some differences between these two types of gained angles, for example, in the High $\times$ High interaction case, when the output is sufficiently small, then the angle between two inputs should be small, however the angle between the Dirac input and the output might be very large.  Hence how to use this gained angle properly is also the key to prove bilinear estimate. 

\vspace{1\baselineskip}
\par The other difficulty is that the null structure inside the DKG system is not strong enough to beat the drawback of  
the fact that we are in a low dimension setting. By ``not strong enough'', heuristically, we mean that
 the extra gain of parallel interaction provided by null structure is only $\theta$, here $\theta$ is
 the angle between the space-time frequencies of two inputs inside
 the null structure. In 3D case, 
the total gain from angular localization and the null structure would just cancel the total loss of modulation
 (distance to characteristic hypersurface) with nothing left thus would cause the logarithm divergence problem
 with respect to modulation. Hence, if the null structure is strong enough, i.e gain more than $\theta$, like the type of null
structure $Q_{0,0}(\cdot, \cdot)$ appeared in the Wave map problem which contributes $\theta^2$. Then the logarithm divergence
problem could be avoided, however we don't have the luxury here. For any two smooth well defined function $u, v$, the bilinear
operator $Q_{0,0}$ is defined as following:
$
Q_{0,0}(\cdot, \cdot): (u, v) \rightarrow  \partial_{t} u \partial_{t} v - \partial_{{i}} u \partial_{{i}} v.
$
\vspace{1\baselineskip}
\par We would like to mention that as the dimension decrease, the  extra
 gain from the angular localization would also decrease. In two spatial dimension
case, the total gained part won't cover the total loss thus the gap phenomena (the gap between the regularity 
predicted by scaling argument  and the regularity needed to have LWP) arise in 2D.
\vspace{1\baselineskip}

Now, we are ready to give an explanation for the case excluded, i.e when $2M \leq m$ and when $M >0$, $m=0$. For the case: $2M < m$, there exists a special scenario of space-time frequencies interaction, which is also the most problematic type.  Let's take $M=0, m > 0$ as an example and this scenario is when two inputs are sufficiently close to the cone while the
 output is also sufficiently close to the hyperboloid. Let's consider the High $\times$ Low interaction case, assume that 
the space time frequencies of inputs are $(|\xi_{1}|, \xi_1)$ and $(|\xi_{2}|,\xi_2) $, $(|\xi_{1}|, \xi_1) \sim 2^{\lambda}, (|\xi_{2}|, \xi_2) \sim {1}$, $\lambda \gg 1$,  then the output frequency is $(\tau, \xi) :=(|\xi_{1}| + |\xi_{2}|,\xi_{1}+ \xi_2)$ . Denote
\[
 \mathbf{C}_{\lambda}:= \{(\tau, \xi): |\tau| - |\xi| = 0, |(\tau, \xi)| \sim 2^{\lambda} \}, \quad \mathbf{H}_{\lambda}
 := \{ (\tau,\xi ): 
|\tau|^2 - |\xi|^2 = m^2, |(\tau, \xi)| \sim 2^{\lambda} \},
\]
therefore we could see that the distance between cone $\mathbf{C}_{\lambda}$  and hyperboloid $\mathbf{H}_{\lambda}$ has
size $m^2 2^{-\lambda}$. Suppose that modulation of output $(\tau, \xi)$ is $d:= ||\tau| - \sqrt{|\xi|^2 + m^2}|$ and the angle
between the two inputs is $\theta:= \angle(\xi_{1},   \xi_{2})$, then we have 
\begin{equation}\label{equation13}
 | \pm d + (1-\cos(\theta)) | \sim m^2 2^{-\lambda}.
\end{equation}
From (\ref{equation13}), we can see that when $d \ll 2^{-\lambda}$, then $\theta$ has size  $m 2^{-\lambda/2}$ and it
doesn't change too much as $d\rightarrow 0$. Hence for this scenario, the gained angle from the null structure
 won't be helpful to cover the loss of modulation. We feel that other ideas and observations might needed to deal with 
this case. For other possible value of $M, m$ such that ${2} M < m$,  one can check that very similar scenario can also happen(or one can see this point from  the discussion on (\ref{equation22080})  in section \ref{nonzeromass}). When $2M=m$, above scenario doesn't happen, however, one potential problem is that the modulation could be very small, while when the nonzero mass term presents, the characteristic hypersurface associated with $\psi_{\pm}$ is the entire hyperboloid instead of one branch (see (\ref{equation90000})), hence there are exist some cases that the type of angle between two inputs is not coincide with the type of null structure (recall the fact that the type of null structure only depends on the types of Dirac inputs), hence we can't actually gain this smallness.
\vspace{1\baselineskip}

While for the last excluded case: $M > 0, m=0$, actually we can gain one degree of angle for the high frequencies part.  Now the problematic part is the low frequency part, without loss of generality, let's assume that $M=1$.  Suppose that the spatial frequencies $|\xi_{1}|$ and $|\xi_{2}|$ are sufficiently small, e.g $|\xi_{1}| \approx |\xi_{2}| \approx C \mu \ll 1$, $C$ is a large constant and  the space-time frequencies is sufficiently close to the hyperboloid, i.e $|\tau_{i} \pm \sqrt{|\xi_{i}|^{2}+1}| \ll 1$ and  the output frequency is also very small $|\tau_{1} +\tau_{2}| + |\xi_{1}+\xi_{2}| \approx \mu \ll 1.$ In this case, the angle between the spatial frequencies of two inputs is possibly big, i.e, \emph{we can't gain smallness from the null structure}. One can verify that in this case the modulation of the output frequency has size $\mu$. Basically, we have to estimate the term of type $S_{\mu,\mu}Q_{a,b} (S^{a}_{1, \bullet < c \mu}\psi_{1}, S^{b}_{1, \bullet < c \mu} \psi_{2})$, here $c$ is a sufficiently small constant. Let's remind the reader that, in this case, the characteristic hypersurface for the output is the entire cone, while for the two Dirac inputs are the entire hyperboloid. Hence, for this type of interaction, heuristically, \emph{we will loss $\mu^{3/2}$ (loss $\mu^{2}$ from $\Box^{-1}$ and gain $\mu^{1/2}$ from the level of regularity itself) which is too much to recover if without gained smallness $\mu$ from null structure}. While this scenario won't happen if $2M \geq m >0$ and $M=m=0$.  As in the massless case, when modulation is small, $|\tau|$ and $|\xi|$ should have comparable size, which is $1$ in the case listed here, hence we could gain $\mu$ from null structure, while if $2M \geq m > 0$, then for this scenario, the modulation of output will have size $1$ and $(-\Box+m^{2})^{-1}$ is like a constant hence won't loss any $\mu$ and the low frequency part won't cause any problem.
 \subsubsection{Strategy of proof}

Inspired from the work of Daniel Tataru \cite{Tataru1,Tataru2} on wave equations and wave maps, the work of Terence Tao 
\cite{Tao, Tao2, Tao1} on wave equations and wave maps and the work of Jacob Sterbenz \cite{Sterbenz1,Sterbenz3,Sterbenz2} on 
general wave equations. Our strategy is to construct appropriate function spaces and then prove that the desired bilinear estimates holds for all types of interactions: Low $\times$High, High $\times$ Low and High $\times$ High interactions of two Dirac fields and  interactions of one Dirac field and one meson field. 

\vspace{1\baselineskip}
While as mentioned in the previous subsubsection, the null
structure inside DKG system is not as strong as the wave map case, the logarithm divergence problem 
arises. However, from the  work of Jacob Sterbenz \cite{Sterbenz1,Sterbenz3,Sterbenz2},
 we could see that if we possess extra regularity with respect to the
angular variables, the parallel interaction can also be eliminated. Heuristically, it's natural
to expect that  such extra regularity condition 
can make up the weakness of the null structures inside the DKG system. The main improvements 
after imposing such extra regularity assumption are 
the improved Strichartz estimates and the
angular concentration lemma. Incorporating those improved estimates into 
the construction of the appropriate function spaces is another goal of this paper. 
\vspace{1\baselineskip}

Among all possible values of $m, M$ such that ${2}M > m>0$ or $M=m=0$, the
the most difficult case is the massless case, i.e, 
$m=M=0$. In the main body part
, we will restrict ourself to the massless case, and in section \ref{nonzeromass}, we will show 
that the method used in the massless case is robust enough
to prove all other possible nonzero mass $m,M$ s.t ${2}M > m > 0$.

\subsection{Outline}
In section \ref{notation}, 
we will introduce notations and  reduce the DKG system into 
a more favorable formulation. In section \ref{nullstructure}, we will 
introduce the null structure 
inside the DKG system and record some lemmas that will be used throughout this paper. 
In section \ref{Functionspace}, we will introduce some frequency localization 
operators and then construct the function spaces where the solution $(\phi, \psi)$ lies in. In section \ref{Bilinear}, we 
will introduce some bilinear decomposition lemmas and angular decomposition lemmas, which will be used to 
prove the bilinear estimates. In section \ref{proof}, 
we will give the detail proof of the main theorems for the massless
 case. In the last section,
 we will show that the method used in section \ref{proof} is robust enough
 to be applied to all other cases.
 
\vspace{1\baselineskip}

\noindent \textbf{Acknowledgement.\,}The author would like thank his advisor Alexandru Ionescu for his continuous encouragement and support, thank Sung-Jin Oh for introducing him this topic and having many useful discussions and thank Jonathan Luk  and Timothy Candy for providing helpful suggestions on revising this paper. He would also like to express gratitude to his friends: Cong Liu, Li Xiao, Jinyi Jin and Yang Zhao etc who be there with him during the summer in Chengdu, where a major part of this work was done. The author was supported by a Yongjin fellowship (2011).

\section{Notations and Preliminaries}\label{notation}

 Given two quantities $A$ and $B$, we use $A\lesssim B$ and $B \gtrsim A$ to denote $A \leq C \cdot B$ for some
 fixed large constant. The notations $A \lesssim_{a} B$ and $B \gtrsim_{a} A$  denote $A\leq C(a)\cdot B$ for some 
constant $C(a)$ that only depends on $a$. For any given quantity $A$, we use $A+$(resp. $A-$) to
 denote some constant that is arbitrarily close to A and larger ( reap. smaller) than $A$, i.e $
\forall \epsilon > 0$ sufficiently small, we may replace  $A+$(resp. $A-$) by $A+\epsilon$ (resp. $A - \epsilon$), 
however we don't require uniformity on $\epsilon$. The expression $a+b+$(reps. $a+b-$) will represent the summation of $a$ and $b+$( resp. $b-$) throughout the whole paper. Recall the following reduced DKG system (\ref{DKG5}) we introduced in the section 
\ref{introduction} :
\begin{equation}\label{DKG1}
\left\{\begin{array}{lr}
-i (\partial_{t} + \alpha \cdot \nabla) \psi + M \beta \psi =   \phi \beta \psi, &\, \\
(-\square + m^{2}) \phi = \langle \beta \psi, \psi \rangle_{\C},&\, \\
\end{array}\right.
\end{equation}
for simplicity, we will abbreviate the  
inner product $\langle\, , \, \rangle_{\C}  $  as $\langle\, , \,\rangle$. To further simplify the DKG system,
 we can diagonalize the matrix operator $- (\partial_{t} + \alpha \cdot \nabla)$ and project $\psi$ into
 eigenspaces corresponding to the eigenvalues $\pm |\xi|$, here $\xi$ is the symbol of operator $-i\nabla$. 
Define the symbols
\begin{displaymath}
\Pi_{\pm}(\xi) = \frac{1}{2}(I \pm \hat{\xi}\cdot \alpha), \quad \hat{\xi} = \frac{\xi}{|\xi|},
\end{displaymath}
and the associated operator as $\Pi_{\pm}(D)$,  here $D = -i \nabla$.  From the definition, it's easy to
 verify that  $\Pi_{\pm}(-\xi) = \Pi_{\mp}(\xi)$ and the following orthogonal decomposition holds
\begin{displaymath}
\psi = \psi_{+} + \psi_{-}, \quad \psi_{\pm} = \Pi_{\pm }(D) \psi,
\end{displaymath}
we'd like to mention that the orthogonality property comes from the fact that $(\alpha \cdot \hat{\xi})^{2} = I$.
\par After applying $\Pi_{\pm}(D)$ operator on the both hands side of Dirac equation and using
 the relation $\Pi_{\pm}(D) \beta = \beta \Pi_{\mp}(D)$, we can further simplify the system (\ref{DKG1}) into  the following formulation

\begin{equation}\label{DKG2}
\left\{\begin{array}{lr}
(-i\partial_{t}  + |D| ) \psi_{+} = -M \beta \psi_{-} + \Pi_{+}(D)(\phi \beta \psi), & \\
(-i \partial_{t} - |D|) \psi_{-} = -M \beta \psi_{+} + \Pi_{-}(D) (\phi \beta \psi), & \\
(-\square + m^{2}) \phi = \langle \beta \psi, \psi \rangle. & \\
\end{array}\right.
\end{equation}
 Denote the operators $L_{\pm} = (- i \partial_{t} \pm |D| )$, 
hence $\square = L_{+}L_{-} = L_{-}L_{+}$.   From now on  and until the section \ref{nonzeromass}, 
let's assume that $ m=M=0$ and focus on the following model:

\begin{equation}\label{masslessDKG}
\mbox{Massless DKG:}\quad  \left\{\begin{array}{lr}
(-i\partial_{t}  + |D| ) \psi_{+} =  \Pi_{+}(D)(\phi \beta \psi), & \\
(-i \partial_{t} - |D|) \psi_{-} = \Pi_{-}(D) (\phi \beta \psi), & \\
-\square \phi = \langle \beta \psi, \psi \rangle, & \\
\end{array}\right.
\end{equation}
with initial data $(\psi_{+,0}, \psi_{-,0}, \phi_{0}, \phi_{1})$, where $\psi_{\pm, 0} = \Pi_{\pm}(D) \psi_{0}$. From the Duhamel formula, we could represent the solution of  (\ref{masslessDKG}) as
\begin{equation}\label{equation2001}
\psi_{\pm} = \tilde{\psi}_{\pm,0} + V_{\pm} \mathcal{N}_{\pm}(\phi, \psi),\quad \phi = \tilde{\phi}_{0} + V \mathcal{N}(\psi\,,\,\psi),
\end{equation}
where $\tilde{\phi}_{0}$ is the linear homogeneous wave solution of $\square \tilde{\phi}_{0} =0$ with initial data $(\phi_{0}, \phi_{1})$ and
\begin{displaymath}
\tilde{\psi}_{\pm, 0} (t)= e^{\mp i t |D|}\psi_{\pm, 0}.
\end{displaymath}
The notation $V f$ denotes the parametrix for the inhomogenous wave 
equation with zero initial data, i.e $ u = V f$ if and only if
\begin{displaymath}
\square u = f, \quad u(0, \cdot) = 0\,\,and\,\, \partial_{t} u(0, \cdot ) = 0.
\end{displaymath}

\par Let $E$ denote any fundamental solution to the homogenous  wave equation, i.e $\square E = \delta$. 
Then we can represent the parametrix operator $V$ via the following formula
\begin{displaymath}
 V(f) = E \ast f - W(E\ast f),
\end{displaymath}
where, for any smooth well defined function $g(t,x)$, $W(g)$ denotes the solution of linear homogeneous
 wave equation with initial data $(g(0,x), \partial_t g (0,x))$. Similarly, the notation $V_{\pm} f$ denotes the parametrix for Dirac equations, i.e $u = V_{\pm}  f$  if and only if u solves the following equation
\begin{displaymath}
L_{\pm} u = f, \quad u(0, \cdot) = 0.
\end{displaymath}
 The notation $\mathcal{N}_{\pm}(\phi, \psi)$ in (\ref{equation2001}) denotes the nonlinearities $\Pi_{\pm}(D)(\phi \beta \psi)$ 
and $\mathcal{N}(\psi,\psi)$ denotes the nonlinearity $\langle \beta \psi, \psi \rangle$. 
After simple Fourier analysis, we could see that the characteristic hypersurface 
 is the light cone for $\phi$, lower cone for $\psi_{+}$ and upper cone for $\psi_{-}$.

\section{Null Structure and Strichartz estimates }\label{nullstructure}

For detail calculation and discussion about the null structures of the 3D DKG system revealed by P. D'Ancona, D. Foschi, and S. Selberg, 
please refer to \cite{D'Ancona1}.  Here, we omit the calculation and only record the necessary parts as granted.
\par
 To see the null structure in the nonlinearity, we first use the orthogonal decomposition $\psi = \psi_{-} + \psi_{+}$  to decompose the nonlinearities $\mathcal{N}_{\pm}(\phi, \psi)$ and $\mathcal{N}(\psi,\psi)$,  and for 
fixed pair of sign $(a, b)$, we define the following bilinear operators:
\begin{displaymath}
Q_{a, b}  : (\psi, \psi^{'}) \longrightarrow \langle \beta \Pi_{a}(D) \psi, \Pi_{b}(D) \psi^{'} \rangle,\quad \tilde{Q}_{a, b} : (\phi , \psi) \longrightarrow \Pi_{a}(D)(\beta \phi \Pi_{b}(D)\psi),
\end{displaymath}
hence
 \begin{equation}\label{nonlinearitydecomposition}
 \mathcal{N}(\psi,\psi) = \sum_{a, b\in \{+,-\} } Q_{a, b}(\psi, \psi),\quad \mathcal{N}_{a}(\phi,\psi) = \sum_{b\in\{+,-\}} \tilde{Q}_{a, b}(\phi, \psi).
\end{equation}
\par After  calculating the symbols associated with the bilinear operators $Q_{a, b}$ and $\tilde{Q}_{a, b}$, we 
can see that the null condition is indeed satisfied. Since
\begin{equation}\label{equation20001}
\widehat{\tilde{Q}_{a, b}(\phi, \psi)} (t, \xi) = \int \beta \Pi_{-a}(\xi) \Pi_{b}(\eta) \hat{\phi}(\xi-\eta) \hat{\psi}(\eta)\, d\,\eta,
\end{equation}
\begin{equation}\label{equation3000}
\widehat{Q_{a, b}(\psi, \psi^{'})}(t,\xi) =\int \langle \beta\Pi_{-b}(\eta)\Pi_{a}(\xi+\eta) \hat{\psi}(t,\xi+\eta), \hat{\psi^{'}}(t,\eta) \rangle d \eta ,
\end{equation}
we can see that  the symbol of $\tilde{Q}_{a, b}$ and $Q_{a, b}$ are $\beta \Pi_{-a}(\xi) \Pi_{b}(\eta)$ and  $\beta\Pi_{-b}(\eta)\Pi_{a}(\xi+\eta)$ respectively.  From \cite{D'Ancona1}[Lemma 2], we have: 
\begin{lemma}\label{angle}
$\Pi_{+}(\xi) \Pi_{-}(\eta) = O(\theta),$ where\,$\theta = \angle (\eta, \xi).$
\end{lemma}
From above lemma, we can see that the symbol of type $\Pi_{-a}(\cdot)\Pi_{b}(\cdot)$ vanishes when the spatial frequencies of two inputs are parallel and in the \emph{same direction} if $a\cdot b > 0$.
While  if $a\cdot b < 0$, the symbol  vanishes
when the spatial frequencies of two inputs are parallel and in the \emph{opposite direction}. 
Those facts are very helpful to analyze the interaction of 
two inputs when both of them are near to their corresponding characteristic hypersurfaces. 
One important observation is that the null structure depends only on spatial frequencies while not on time, 
thus we are free to use the information about null structures when the time variable is fixed.

\begin{lemma}[Angular frequency localized two scale Strichartz estimate \cite{Sterbenz3}]\label{twoscale}

Assume that $f_{1,N}$ has unit frequency and its angular frequency localized around $N$ in $3$ 
spatial dimension space. By angular localized around $N$, we mean that $\mathcal{P}_{N} f_{1, N} = f_{1,N}$, 
and here $\mathcal{P}_{N}$ is the spectrum projection operator associated with $\Delta_{\mathbb{S}^{2}}$. 
Let $0 < \mu \lesssim 1$ be a given constant and $\{ Q_{\alpha} \}$ be a partition of $\mathbb{R}^{3}$ into 
cubes of side length $\sim \frac{1}{\mu}$. Then for every $0 < \eta$, there is a $C_{\eta}$ and 
$r_{\eta} > 4$ depending on $\eta$, such that $r_{\eta} \rightarrow 4$ as $\eta \rightarrow 0$ 
such that the following estimates hold:
\begin{enumerate}
\item[(i)] Two-Scale angular frequency localized Strichartz estimate
\begin{equation}
\|  (\sum_{\alpha} \| e^{it |\nabla| } f_{1, N} \|_{L^{2}(Q_{\alpha})}^{r_{\eta}} )^{\frac{1}{r_{\eta}}} \|_{L^{2}_{t}} \lesssim C_{\eta} \mu^{-1/2-2\eta} N^{1/2+\eta} \| f_{1, N} \|_{L^{2}_{x}}.
\end{equation}
\item[(ii)] Angular frequency localized Strichartz estimate
\begin{equation}\label{equation20002}
  \| e^{i t |\nabla|} f_{1,N} \|_{L^{2}_{t} L^{r_{\eta}}_{x}} \lesssim N^{1/2+\eta} \| f_{1,N} \|_{L^{2}_{x}}.
\end{equation}
\end{enumerate}
\end{lemma}
As a byproduct of angular frequency localized Strichartz estimate (\ref{equation20002}), we could derive the following improved Strichartz estimate very easily:
\begin{equation}\label{equation30}
  \| e^{i t |\nabla|} f \|_{L^{2}_{t} L^{4+}_{x}} \lesssim \| \langle \Omega \rangle^{\frac{1}{2}+} f \|_{L^{2}_{x}}.
  \end{equation}

\section{Function Spaces}\label{Functionspace}

In this section, most of  notations are consistent with the notations used in \cite{Sterbenz2}.
 For simplicity, we only record those necessary parts and refer the readers to \cite{Sterbenz1, Sterbenz3, Sterbenz2} for detail.
We strongly recommend readers to read those papers as warming up to better understand this paper.

\par Let $\phi$ be a one dimensional smooth bump function, 
such that $\phi (s) = 1 $ for $|s|\leq 1$ and $\phi(s) = 0$ for $|s| \geq 2$.
 For $\lambda \in 2^{\mathbb{Z}}$,  define the dyadic scaling of
 $\phi$ by $\phi_{\lambda}(s) = \phi(s/\lambda)$. We can define the localization operators
 which are localized with respect to the spatial frequency,
 space-time frequency, distance to the cone (modulation),  
distance to the lower cone and distance to the upper cone by the following localization functions:
\begin{equation}
p_{\lambda}(\xi) = \phi_{2\lambda}(|\xi|) - \phi_{1/2 \lambda} (|\xi|),\quad
s_{\lambda} (\tau, \xi) = \phi_{2\lambda}(|(\tau, \xi)|) - \phi_{1/2 \lambda}( | (\tau, \xi) |  ),
\end{equation}
\begin{equation}
c_{d}(\tau, \xi) = \phi_{2 d} ( |\tau | - |\xi | ) - \phi_{1/2 d} (|\tau| - |\xi|),
\quad
c^{+}_{d}(\tau, \xi) = \phi_{2 d} ( \tau + |\xi | ) - \phi_{1/2 d} (\tau + |\xi|),
\end{equation}
\begin{equation}
c^{-}_{d}(\tau, \xi)= \phi_{2 d}  ( \tau  - |\xi | ) - \phi_{1/2 d} ( \tau - |\xi|).
\end{equation}
We define the associated Littlewood-Paley type localization operators
 $P_{\lambda}, S_{\lambda}, C_{d}, C^{+}_{d}, C_{d}^{-}$  which have symbols $p_{\lambda}, s_{\lambda}, c_{d}, c_{d}^{+}, c_{d}^{-}$ respectively. Denote $S_{\lambda, d} = S_{\lambda} C_{d}$ , $S_{\lambda, d}^{\pm} = S_{\lambda} C^{\pm}_{d}$  and
\begin{equation}
S_{\lambda, \bullet \leq d} = \sum_{ \delta, \delta \leq d} S_{\lambda, \delta},\quad S_{\lambda, \bullet \leq d}^{\pm} = \sum_{ \delta, \delta \leq d} S_{\lambda, \delta}^{\pm}, \quad S_{\lambda,\bullet \geq d} = \sum_{ \delta, \delta \geq d} S_{\lambda, \delta},\quad S_{\lambda, \bullet \geq d}^{\pm} = \sum_{ \delta, \delta \geq d} S_{\lambda, \delta}^{\pm}.
\end{equation}

For simplicity and without cause any confusion, if 
we use  $\pm$ in the upper subscript, it will represent with respect 
to the lower cone for $``+''$ (corresponds to $\psi_{+}$) and to upper
 cone for$ ``-''$ (corresponds to $\psi_{-}$)  throughout this paper.

Besides aforementioned dyadic projection operators, 
there is another type of localization operators which are very important
 to the bilinear estimates and are so called spatial angular localization operators.

\par For any given small number $0 < \eta\lesssim 1$, we decompose the unit sphere $\mathbb{S}^{2}$
 in $\mathbb{R}^{3}$ into bounded overlapping angular sectors. And each angular sector has
 angular size $\eta$. We label the corresponding sectors by their angles $\omega = \xi/|\xi|$ 
and denote them as $b^{\omega}_{\eta}$. The angular localization operator 
associated with the symbol $b^{\omega}_{\eta}$ is denoted by $B_{\eta}^{\omega}$. Define
\begin{equation}
S^{\omega}_{\lambda, d} = B^{\omega}_{(d/\lambda)^{1/2}} P_{\lambda} S_{\lambda, d},  \quad S^{\omega}_{\lambda, \bullet \leq d} = B^{\omega}_{(d/\lambda)^{1/2}} P_{\lambda} S_{\lambda, \bullet \leq d},
\end{equation}
and from definition, it's easy to see that $S^{\omega}_{\lambda, d}$ projects space-time frequency to a  
parallelepiped of size $\lambda \times \sqrt{\lambda d} \times \sqrt{\lambda d} \times d$. 
We have the following lemma regarding on the boundedness of above localization operators:

\begin{lemma}[\cite{Sterbenz1}]
\begin{enumerate}
\item[(i)] The following multipliers are given by $L_{t}^{1}L^{1}_{x}$ kernels and are uniformly bounded in $L_{t}^{1}L^{1}_{x}$: $\lambda^{-1}\nabla S_{\lambda}$, $B^{\omega}_{(d/\lambda)^{1/2}} P_{\lambda}$, $S_{\lambda, d}^{\omega}$, $(\lambda d) V S^{\omega}_{\lambda, d}$ and  $(\lambda d) V_{\pm} S^{\omega,\pm}_{\lambda, d}$ and also those operators  are bounded in mixed Lebesgue spaces $L^{q}_{t}L^{r}_{x}$.

\item[(ii)] The following multipliers are uniformly bounded in the $L^{q}_{t}L^{2}_{x}$ spaces for $1\leq q \leq +\infty$: $S_{\lambda, d}$, $S_{\lambda,\bullet \leq d}$, $S_{\lambda, d}^{\pm}$ and $S_{\lambda, \bullet \leq d}^{\pm}$.

\end{enumerate}

\end{lemma}
With above defined localization operators, we are ready to define the function spaces. Define
\[
F_{\Omega, \lambda} =  [\langle \Omega\rangle^{-1}(X^{1/2,1}_{\lambda }+ Y_{\lambda}) \cap S_{\lambda} (L^{\infty}_{t}L^{2}_{x})\cap Z_{\Omega, \lambda}],\,\,F_{\Omega, \lambda}^{\pm} = [ \langle \Omega\rangle^{-1}(X^{1/2,1}_{\lambda,\pm} + Y_{\lambda}^{\pm}) \cap S_{\lambda} (L^{\infty}L^{2})\cap Z_{\Omega, \lambda}^{\pm}],
\]
here $X^{1/2, 1}_{\lambda}$ is the space of functions with space-time 
frequency localized in the support of $s_{\lambda}(\tau, \xi)$ and equipped with the norm
\begin{displaymath}
\| u \|_{X^{1/2,1}_{\lambda}} := \sum_{d:\,\,d \in 2^{\mathbb{Z}}} d^{1/2} \| S_{\lambda, d} u \|_{L^{2}_{t}L^{2}_{x}},
\end{displaymath}
and $Y_{\lambda} = \lambda V S_{\lambda}(L^{1}_{t}L^{2}_{x})$ is the space of functions with  space-time frequency localized in the support of $s_{\lambda}(\tau, \xi)$ and equipped with the norm
\begin{displaymath}
\| u \|_{Y_{\lambda}} := \lambda^{-1} \| \square S_{\lambda} u \|_{L^{1}_{t}L^{2}_{x}}.
\end{displaymath}
The function space $Z_{\Omega, \lambda}$ is defined by the following norm:
\[
\| f\|_{Z_{\Omega,\lambda}} := \lambda^{-\frac{1}{2}+\frac{3}{p}} \sum_{d\lesssim \lambda}
 \Big( \frac{d}{\lambda}\Big)^{\frac{1.1}{p}} \int \sup_{\omega} \| S^{\omega}_{\lambda, d} f \|_{L^{p}_{x}} \, d\, t,\quad \emph{here $p\in (5,6)$ is a fixed constant},
\]
bu using Sobolev embedding lemma \ref{Sobolev} and the angular concentration lemma \ref{Angular}, we have
\[
\| f\|_{Z_{\Omega, \lambda}} \lesssim \lambda^{-\frac{1}{2}+\frac{3}{p} + 3(\frac{1}{2}- \frac{1}{p})} \sum_{d \lesssim \lambda}  \Big( \frac{d}{\lambda}\Big)^{(- \frac{1}{p}+ 1-)+\frac{1.1}{p}}  \lambda^{-1} d^{-1} \| S_{\lambda, d} \Box \langle \Omega\rangle f \|_{L^{1}_{t} L^{2}_{x}},
\]
\[
\lesssim \lambda^{-1}\| S_{\lambda} \Box \langle \Omega \rangle  f\|_{L^{1}_{t}L^{2}_{x}} \lesssim \| S_{\lambda} \langle \Omega \rangle f \|_{Y_{ \lambda}},
\]
hence
\[
\langle \Omega \rangle^{-1} Y_{\lambda}  \subseteq Z_{\Omega, \lambda}, \quad F_{\Omega, \lambda} =  [\big( \langle \Omega\rangle^{-1}X^{1/2,1}_{\lambda }\cap Z_{\Omega, \lambda}+ \langle \Omega\rangle^{-1} Y_{\lambda} \big) \cap S_{\lambda} (L^{\infty}_{t}L^{2}_{x})],
\]

\begin{equation}
\sup_{d\lesssim \lambda} \Big( \frac{d}{\lambda}\Big)^{\frac{1.1}{p}}  \| \sup_{\omega} \| S^{\omega}_{\lambda, d} f \|_{L^{p}_{x}}\|_{L^{1}_{t}} \lesssim \lambda^{\frac{1}{2}-\frac{3}{p}} \| f \|_{F_{\Omega, \lambda}}.
\end{equation}

Very similarly, we can define function space $Y_{\lambda}^{\pm}: = V_{\pm} S_{\lambda}( L^{1}_{t}L^{2}_{x}) $ and $Z_{\Omega, \lambda}^{\pm}$ as  the space of functions whose space-time frequencies are localized in the support of $s_{\lambda}(\tau, \xi)$ and equipped with the following norms respectively
\begin{displaymath}
\| u \|_{Y^{\pm}_{\lambda}} = \| L_{\pm} S_{\lambda} u \|_{L^{1}_{t}L^{2}_{x}}, \quad  \| f\|_{Z_{\Omega,\lambda}^{\pm}} := \lambda^{-\frac{1}{2}+\frac{3}{p}} \sum_{d\lesssim \lambda}
 \Big( \frac{d}{\lambda}\Big)^{\frac{1.1}{p}} \int \sup_{\omega} \| S^{\omega, \pm}_{\lambda, d} f \|_{L^{p}_{x}} \, d\, t.
\end{displaymath}

For the inhomogeneous angular derivative $\langle \Omega\rangle$, one of the key properties that we'd like to mention here is that $\langle \Omega \rangle$ commute with the Fourier transform 
operator, more precisely, for any $s\in \mathbb{R}$,
\[
\widehat{\langle \Omega\rangle^{s}} f = \langle \Omega\rangle^{s} \widehat{f} \]
and when $s\geq 0$, heuristically, one can distribute $\langle \Omega\rangle^{s}$ as the usual chain rule.

\par 
Notice that $Y_{\lambda}$ norm and $Y_{\lambda}^{\pm}$ norm are defined inspire from the Duhamel formula
 of Wave equation and Dirac equation respectively, and 
the reason why we incorporate $L^{\infty}_{t}L^{2}_{x}$ norm into the definition is to measure the part
of space-time frequencies that lie on the characteristic hypersurfaces, as the $L^{\infty}_{t}L^{2}_{t}$
 norm of the part that outside the characteristic hypersurface is dominated by $Y_{\lambda}$ norm and $X^{1/2,1}_{\lambda}$ 
norm (resp. $Y_{\lambda}^{\pm}$ norm and $X^{1/2,1}_{\lambda,\pm}$), one can verify
 this fact
from the Duhamel formula and the Stricharz estimates.

Now we can glue above defined norm for dyadic pieces together to define the main function spaces, which are defined by the following norms: 
\begin{equation}
\| \phi \|_{F^{s}_{\Omega}}^{2} = \sum_{\lambda \in 2^{\mathbb{Z}}} \lambda^{2 s} \| S_{\lambda} \phi \|_{F_{\Omega,\lambda}}^{2},\quad \| \psi_{\pm} \|_{F^{r, \pm}_{\Omega}}^{2} = \sum_{\lambda \in 2^{\mathbb{Z}}} \lambda^{2 r} \| S_{\lambda} \psi_{\pm} \|_{F_{\Omega,\lambda}^{\pm}}^{2},\quad s > 1/2, r > 0,
\end{equation}
\begin{equation}
 \| \phi \|_{F^{{1}/{2}}_{\Omega}} =  \sum_{\lambda \in 2^{\mathbb{Z}}} \lambda^{\frac{1}{2}} \| S_{\lambda} \phi \|_{F_{\Omega,\lambda}},\quad \| \psi_{\pm} \|_{F^{\pm}_{\Omega}} = \sum_{\lambda \in 2^{\mathbb{Z}}}  \| S_{\lambda} \psi_{\pm} \|_{F_{\Omega, \lambda}^{\pm}},
\end{equation}
\begin{equation}
 \| \psi\|_{\widetilde{F}_{\Omega}} = \| \psi_{+}\|_{F^{+}_{\Omega}} + \| \psi_{-}\|_{F^{-}_{\Omega}} ,\quad \| \psi\|_{\widetilde{F}^{r}_{\Omega}} = \| \psi_{+}\|_{F^{r,+}_{\Omega}} + \|\psi_{-}\|_{F^{r,-}_{\Omega}}, r > 0 .
\end{equation}
Recall the fact that $\Pi_{\pm}(D) \Pi_{\mp}(D) = 0 $, thus from above definitions, we could easily see that 
$$\|\psi_{\pm} \|_{\widetilde{F}_{\Omega}} = \| \psi_{\pm} \|_{F^{\pm}_{\Omega}}, \quad \| \psi_{\pm}\|_{\widetilde{F}^{r}_{\Omega}} = \| \psi_{\pm}\|_{F^{r,\pm}_{\Omega}}.$$

From the trace method which allows one to transfer the 
estimate from the space of solutions to the homogeneous wave equation
 to the $X^{1/2,1}_{\lambda}$ space,  we can derive the following estimates from the improved Strichartz estimate (\ref{equation30}),
\begin{equation}
\| S_{1} u \|_{L^2_{t}L^{4^{+}}_{x}} \lesssim
\|\langle \Omega\rangle^{\frac{1}{2}+} u \|_{X^{1/2,1}_{1}} + 
\| S_{1}(\langle \Omega\rangle^{\frac{1}{2}+}  u) \|_{L^{\infty}_{t}L^{2}_{x}}.
\end{equation}

From  the Duhamel formula and the improved Strichartz estimate (\ref{equation30}), we also have the following estimate: \begin{equation}
\| S_{1} u \|_{L^2_{t}L^{4^{+}}_{x}} \lesssim
\| \langle \Omega\rangle^{\frac{1}{2}+}  u \|_{Y_{1}} + \| S_{1}(\langle \Omega\rangle^{\frac{1}{2}+}  u) \|_{L^{\infty}_{t}L^{2}_{x}}.
\end{equation}
Hence from the definition of function space $F_{\Omega, \lambda}$, we have
\begin{equation}\label{strichartzembedding}
\| S_{1}  u \|_{L^2_{t}L^{4^{+}}_{x}} \lesssim \| \langle \Omega\rangle^{- (\frac{1}{2}-)} u \|_{F_{\Omega, 1}}.
\end{equation}

After applying scaling argument, we could get the following estimate for general space-time frequency localized function
\begin{equation}\label{equation2}
\| S_{\lambda} f \|_{L^2_{t}L^{4^{+}}_{x}} \lesssim \lambda^{\frac{1}{4}{+}} \| \langle \Omega\rangle^{-(\frac{1}{2}-)} u \|_{F_{\Omega, \lambda}}.
\end{equation}

As $X_{\lambda}^{1/2,1} \subset S_{\lambda}(L^{\infty}_{t}L^{2}_{x})$, naturally
 we have $ \square X_{\lambda}^{1/2,1} \subset \square S_{\lambda}( L^{\infty}_{t}L^{2}_{x})$, and by duality,
 we could derive
\begin{equation}
V S_{\lambda}(L^{1}_{t}L^{2}_{x})\subset (\square X_{\lambda}^{1/2,1} )^{'} = (\lambda X_{\lambda}^{-1/2,1})^{'}= \lambda^{-1} X^{1/2, \infty}_{\lambda},
\end{equation} 
recall that we also have an obvious inclusion relation: $ X^{1/2, 1}_{\lambda} \subset  X^{1/2, \infty}_{\lambda}$,
hence, we have
\begin{equation}\label{equation3}
d^{1/2} \|S_{\lambda, d} f \|_{L^2_{t}L^2_{x}} \leq \| \langle \Omega\rangle^{-1} f \|_{F_{\Omega, \lambda}},\quad for\,\, d \in 2^{\mathbb{Z}},0 < d \leq \lambda.
\end{equation}
Similar argument could give us an analogue of (\ref{equation3}):
\begin{equation}
 d^{1/2} \|S_{\lambda, d}^{\pm} f_{\pm} \|_{L^2_{t}L^2_{x}} \leq \| \langle \Omega\rangle^{-1}   f \|_{F_{\Omega, \lambda}^{\pm}},\quad for\,\, d \in 2^{\mathbb{Z}},0 < d \leq \lambda.
\end{equation}


\begin{lemma}[Sobolev Embedding Estimate]\label{Sobolev}
Let $f$ be a test function on $\mathbb{R}^{3}$, then one has the following frequency localized estimate:
\begin{equation}
\| B^{\omega}_{\eta} P_{1} f \|_{L^{p}_{x}} \lesssim \eta^{2(\frac{1}{r}- \frac{1}{p})} \| f \|_{L^{r}_{x}},
\end{equation}
and also by scaling argument, we could derive
\begin{equation}
\| B^{\omega}_{\eta} P_{\lambda} f \|_{L^{p}_{x}} \lesssim \eta^{2(\frac{1}{r}- \frac{1}{p})} \lambda^{3(\frac{1}{r}- \frac{1}{p})} \| f \|_{L^{r}_{x}}
\end{equation}
\end{lemma}

\begin{lemma}[Angular Concentration Estimate]\label{Angular}
For any test function $f$ defined on $\mathbb{R}^{3}$, and any $2\leq p < +\infty$ one has the following estimate:
\begin{equation}\label{weightedangular}
\sup_{\omega} \| B^{\omega}_{\eta}\, f \|_{L^{p}_{x}} \lesssim \eta^{s} \| \langle \Omega \rangle^{s} f \|_{L^{p}_{x}},
\end{equation}
where $0\leq s < \frac{2}{p}$. In particular, if $f$ is spherically symmetric, we could have the following estimate
\begin{equation}
\sup_{\omega} \| B^{\omega}_{\eta} f \|_{L^{p}_{x}} \leq \eta^{2/p} \| f \|_{L^{p}_{x}}.
\end{equation}
\end{lemma}
\begin{remark}
The angular concentration estimate is derived by interpolation between $
L^{\infty} \rightarrow L^{\infty}$ bound and $L^{2} \rightarrow L^{2} $ bound. When $f$ is radial, we don't need to use
Sobolev embedding lemma on the sphere which is used in the proof of (\ref{weightedangular}), hence we could reach the endpoint. Notice the fact that in Sobolev embedding estimate \ref{Sobolev}, we don't require $f$ to be localized in an angular sector. Since $B^{\omega}_{\eta}B^{\omega}_{\eta} f \sim B^{\omega}_{\eta} f$, we can
apply Sobolev embedding estimate first and then apply angular concentration estimate to utilize most information about the angular localization. 
\end{remark}

\section{Bilinear Decomposition and Angular Decomposition}\label{Bilinear}

In this section, we will decompose all types of bilinear forms and then introduce the angular decompositions lemmas  for those terms when both inputs and outputs are very close to the cone. In later proof of main theorem, we will see the application of  the all types of decomposition we did in this section. The content of discussion in this section would be 
very similar to the section 6 of \cite{Sterbenz2}, as in the massless case, the characteristic hypersurfaces are still light cone. 

Before introducing all types of decomposition, we emphasize that there are two facts better be remembered: $(i):$ the characteristic hypersurface of $\psi_{+}$ is the lower cone, of $\psi_{-}$ is the upper cone and of $\phi$ is the entire light cone. $(ii):$ If two input frequencies are $(\tau_{1}, \xi_{1})$ and $(\tau_{2},\xi_{2})$, then the output frequency associated with $Q_{a,b}$ is $(\tau_{1}-\tau_{2}, \xi_{1}-\xi_{2})$ due to complex conjugation, the output frequency associated with $\tilde{Q}_{a,b}$ is $(\tau_{1} + \tau_{2}, \xi_{1}+\xi_{2})$, see (\ref{equation20001}) and (\ref{equation3000}).

\par Let's first assume that $\mu \leq c \lambda$ for a sufficiently small constant $c$, $c\ll 1$ and consider the High $\times$ High interaction case, generally we have the following type:
\begin{displaymath}
 S_{\mu} T_{a, b} (S_{\lambda} u, S_{\lambda} v), \quad T_{a, b}\in \{Q_{a, b}, \tilde{Q}_{a,b}\},  a,b\in \{+, -\}.\end{displaymath} 
After fixing constants $c_{1}, c_{2}$ and $c_{3}$ s.t $12 \leq c_{3}+8\leq c_{2}+4\leq c_{1}$, we can decompose above High $\times$ High interaction into the 
following forms by utilizing the fact that as the output is localized around $\mu$, then the two input frequencies can't too far from each other. When $a\cdot b > 0$, then we have the following decomposition:
\begin{equation}\label{equation30005}
S_{\mu} Q_{a, b} (S_{\lambda} u, S_{\lambda} v) = A_{1}\,+\,B_{1}\,=\,S_{\mu} Q_{a,b} (S_{\lambda, \bullet \geq c_{2}\mu}^{a} u, S^{b}_{\lambda, \bullet \geq c_{3} \mu} v)
+ S_{\mu} Q_{a, b} (S_{\lambda, \bullet < c_{2} \mu}^{a} u, S_{\lambda, \bullet < c_{1} \mu}^{b} v),
 \end{equation}
 when $a\cdot b < 0$, we have
 \begin{equation}\label{equation71004}
S_{\mu} Q_{a, b} (S_{\lambda} u, S_{\lambda} v) = S_{\mu} Q_{a,b} (S_{\lambda, \bullet \geq c \lambda}^{a} u, S^{b}_{\lambda} v)
+ S_{\mu} Q_{a, b} (S_{\lambda, \bullet \leq c\lambda}^{a} u, S_{\lambda, \bullet \geq c \lambda}^{b} v).
 \end{equation}
While for all possible sign of $a$ and $b$, we have following decomposition for $\tilde{Q}_{a,b}$:
\[
S_{\mu} \tilde{Q}_{a, b} (S_{\lambda} u, S_{\lambda} v) = \,S_{\mu} \tilde{Q}_{a,b} (S_{\lambda, \bullet \geq c_{2}\mu}^{-b} u, S^{b}_{\lambda, \bullet \geq c_{3} \mu} v) + S_{\mu} \tilde{Q}_{a, b} (S_{\lambda, \bullet < c_{2} \mu}^{-b} u, S_{\lambda, \bullet < c_{1} \mu}^{b} v)
\]

\[
= S_{\mu} \tilde{Q}_{a,b} (S_{\lambda, \bullet \geq c_{2}\mu} u, S^{b}_{\lambda, \bullet \geq c_{3} \mu} v) + S_{\mu} \tilde{Q}_{a, b} (S_{\lambda, \bullet < c_{2} \mu}^{-b} u, S_{\lambda, \bullet < c_{1} \mu}^{b} v) + S_{\mu} \tilde{Q}_{a,b} (S_{\lambda, \bullet < c_{2}\mu}^{b} u, S^{b}_{\lambda, \bullet \geq c_{3} \mu} v)
\]
\[
= A_{2}+ B_{2}+ C_{2}:= S_{\mu} \tilde{Q}_{a,b} (S_{\lambda, \bullet \geq c_{2}\mu} u, S^{b}_{\lambda, \bullet \geq c_{3} \mu} v)
\]
\begin{equation}\label{equation71000}
 + S_{\mu} \tilde{Q}_{a, b} (S_{\lambda, \bullet < c_{2} \mu}^{-b} u, S_{\lambda, \bullet < c_{1} \mu}^{b} v) + S_{\mu} \tilde{Q}_{a,b} (S_{\lambda, \bullet < c_{2}\mu}^{b} u, S^{b}_{\lambda, \bullet \geq c \lambda} v).
\end{equation}

\par Terms $A_{1}$, $A_{2}$ and $C_{2}$  are relatively easier to deal with when compares with the terms $B_{i}$, $i \in\{1,2\}$. While in order to see the angular structure
inside the term $B_{i}$, we have to do further decomposition:
\[
B_{1} = \scriptsize \sum_{\begin{array}{c}
                                                                                                  d, \delta_1, \delta_2\\
\delta_1 < c_{2} \mu, \delta_2 < c_{1}\, \mu\\
                                                                                               \end{array}} \normalsize
{S}_{\mu, d} Q_{a, b} (S_{\lambda, \delta_1}^a u , S_{\lambda, \delta_2}^b v)
 = B_{\textrm{I}}^{1} + B_{\textrm{II}}^{1} + B_{\textrm{III}}^{1}= \sum_{d\leq \mu} {S}_{\mu, d} Q_{a, b} (S_{\lambda, \bullet < d}^{a} u  ,S_{\lambda, \bullet < d}^{b} v) \]
\begin{equation}\label{equation1008}
 + \sum_{ d < c_{2} \mu} {S}_{\mu, \bullet \leq \min\{d, \mu\}} Q_{a, b}( S_{\lambda, d}^{a} u, S_{\lambda, \bullet < d}^{b} v) + 
\sum_{ d < c_{1} \mu} {S}_{\mu, \bullet \leq \min\{d, \mu\}} Q_{a, b} (S_{\lambda, \bullet \leq \min\{d, c_{2}\mu\}}^{a} u, S_{\lambda, d}^{b} v).
\end{equation}

\[
B_{2} = \scriptsize \sum_{\begin{array}{c}
                                                                                                  d, \delta_1, \delta_2\\
\delta_1 < c_{2} \mu, \delta_2 < c_{1}\, \mu\\
                                                                                               \end{array}} \normalsize
{S}_{\mu, d}^{a} \tilde{Q}_{a, b} (S_{\lambda, \delta_1}^{-b} u , S_{\lambda, \delta_2}^b v)
 = B_{\textrm{I}}^{2} + B_{\textrm{II}}^{2} + B_{\textrm{III}}^{2}=  \sum_{d\leq \mu} {S}_{\mu, d}^{a} \tilde{Q}_{a, b} (S_{\lambda, \bullet < d}^{-b} u  ,S_{\lambda, \bullet < d}^{b} v)
\]
\begin{equation}\label{equation1004}
 + \sum_{ d < c_{2} \mu} {S}_{\mu, \bullet \leq \min\{d, \mu\}}^{a} \tilde{Q}_{a, b}( S_{\lambda, d}^{-b} u,S_{\lambda, \bullet < d}^{b} v) + \sum_{d < c_{1} \mu} {S}_{\mu, \bullet \leq \min\{d, \mu\}}^{a} \tilde{Q}_{a, b} (S_{\lambda, \bullet \leq \min\{d, c_{2}\mu\}}^{-b} u, S_{\lambda, d}^{b} v).
\end{equation}

Following the similar argument of \cite{Sterbenz2}[section 6], we could derive the following High $\times$ High angular decomposition lemma for the terms $B_{\textrm{I}}^{i}$, $B_{\textrm{II}}^{i}$ and $B_{\textrm{III}}^{i}$, $i\in\{1,2\}$.
\begin{lemma}[High $\times$ High angular decomposition]\label{highhighangular}
For the bilinear form expressions:
\begin{displaymath}
S_{\mu, d} \, Q_{a, b} (S_{\lambda,\,\bullet < d }^{a} u\,,\, S_{\lambda, \,\bullet <d}^{b} v),\quad S_{\mu, d}^{a} \, \tilde{Q}_{a, b} (S_{\lambda,\,\bullet < d}^{-b} u\,,\, S_{\lambda, \,\bullet < d}^{b} v),
\end{displaymath}
we could have the following angular decomposition
\begin{displaymath}
S_{\mu, d}^{\pm} \, Q_{a,b} (S_{\lambda,\,\bullet <d}^{a} u\,,\, S_{\lambda, \,\bullet <
d}^{b} v) \,\,=\end{displaymath}
\begin{displaymath}
\sum_{\tiny \begin{array}{c}
\omega_{1}, \omega_{2},\omega_{3}\\
|\omega_{1} \mp a \cdot \omega_{2}| \sim (\frac{d}{\mu})^{1/2}\\
|\omega_{2}- a b \cdot \omega_{3}| \sim (\frac{d}{\mu})^{1/2}\\
\end{array}
} S^{\omega_{1}, \pm}_{\mu, d} Q_{a, b} (B^{\omega_{2}}_{(\frac{d}{\mu})^{1/2}} S_{\lambda,\,\bullet <  d}^{a} u\,,\,  B^{\omega_{3}}_{(\frac{d}{\mu})^{1/2}}\, \, S_{\lambda, \,\bullet < d}^{b} v)
\end{displaymath}

\begin{displaymath}
S_{\mu, d}^{a}  \tilde{Q}_{a, b} (S_{\lambda,\,\bullet < d}^{-b} u\,,\, S_{\lambda, \,\bullet <  d}^{b} v) \,\,=\end{displaymath}
\begin{displaymath}
 \sum_{\tiny \begin{array}{c}
\omega_{1}, \omega_{2},\omega_{3}\\
|\omega_{1} - a b\cdot \omega_{3}| \sim (\frac{d}{\mu})^{1/2}\\
|\omega_{2}+  \omega_{3}| \sim (\frac{d}{\mu})^{1/2}\\
\end{array}
}\normalsize S^{\omega_{1}, a}_{\mu, d} \tilde{Q}_{a, b}  (B^{\omega_{2}}_{(\frac{d}{\mu})^{1/2}} S_{\lambda,\,\bullet < d}^{-b} u\,,\,  B^{\omega_{3}}_{(\frac{d}{\mu})^{1/2}}\, \, S_{\lambda, \,\bullet < d}^{b} v).
\end{displaymath}
Similar angular decompositions hold for the bilinear forms $B_{\textrm{II}}^{i}$ and $B_{\textrm{III}}^{i}$, $i\in\{1,2\}$.
\end{lemma}

Now let's proceed to  consider the Low $\times$ High interaction case  under the assumption that $\mu \leq c \lambda$, $c\ll 1$.  Generally we  have the following type of interaction:
\begin{displaymath}
 S_{\lambda} T_{a, b} (S_{\mu} u, S_{\lambda} v),  \quad T_{a, b}\in \{Q_{a, b}, \tilde{Q}_{a,b}\},  a,b\in \{+, -\}.
\end{displaymath}

Very similar to  the High $\times$ High interaction case, we can decompose above general bilinear form into the following forms by utilizing the fact that the output frequency and the high input frequency can't be too far from each other.  Hence
\[
S_{\lambda} Q_{a, b} (S_{\mu} u, S_{\lambda} v) = S_{\lambda, \bullet \geq c_{3}\mu}^{-b} Q_{a,b} (S_{\mu} u, S_{\lambda, \bullet \geq c_{2} \mu}^{b} v)
+\, S_{\lambda, \bullet \leq c_{1} \mu}^{-b} Q_{a, b}( S_{\mu} u, S_{\lambda, \bullet < c_{2} \mu}^{b} v)
\]
\[
= S_{\lambda, \bullet \geq c_{3}\mu} Q_{a,b} (S_{\mu} u, S_{\lambda, \bullet \geq c_{2} \mu}^{b} v) + S_{\lambda, \bullet \leq c_{3}\mu}^{b} Q_{a,b} (S_{\mu} u, S_{\lambda, \bullet \geq c_{2} \mu}^{b} v)
+\, S_{\lambda, \bullet \leq c_{1} \mu}^{-b} Q_{a, b}( S_{\mu} u, S_{\lambda, \bullet < c_{2} \mu}^{b} v)
\]
\[
= \tilde{A}_{1} + \tilde{B}_{1} + \tilde{C}_{1} := S_{\lambda, \bullet \geq c_{3}\mu} Q_{a,b} (S_{\mu} u, S_{\lambda, \bullet \geq c_{2} \mu}^{b} v) +
 \]
\begin{equation}\label{equation60000}
  S_{\lambda, \bullet < c_{1} \mu}^{-b} Q_{a, b}( S_{\mu} u, S_{\lambda, \bullet < c_{2} \mu}^{b} v) +
S_{\lambda, \bullet \leq c_{3}\mu}^{b} Q_{a,b} (S_{\mu} u, S_{\lambda, \bullet \geq c \lambda}^{b} v).
\end{equation}
While for $\tilde{Q}_{a,b}$, sign matters here, when $a\cdot b < 0$ we have
\begin{equation}\label{equation30001}
S_{\lambda} \tilde{Q}_{a, b} (S_{\mu} u, S_{\lambda} v)  =\, \, S_{\lambda} \tilde{Q}_{a,b} (S_{\mu} u, S_{\lambda, \bullet \geq c \lambda}^{b} v)+ S_{\lambda, \bullet \geq c \lambda}^{a} \tilde{Q}_{a, b}( S_{\mu} u, S_{\lambda, \bullet < c  \lambda}^{b} v),
\end{equation}
and when $a\cdot b > 0$, we have
\begin{equation}\label{equation71001}
S_{\lambda} \tilde{Q}_{a, b} (S_{\mu} u, S_{\lambda} v) = \tilde{A}_{2} + \tilde{B}_{2} := S_{\lambda, \bullet \geq c_{3}\mu}^{a} \tilde{Q}_{a,b}(S_{\mu} u, S_{\lambda, \bullet \geq c_{2}\mu}^{b} v ) + S_{\lambda, \bullet < c_{1}\mu}^{a} \tilde{Q}_{a, b}(S_{\mu} u, S_{\lambda, \bullet < c_{2}\mu}^{b} v).
\end{equation}
For the High $\times$Low interaction case, due to symmetry between inputs of $Q_{a,b}$, we can switch role of $u$ and $v$ inside $Q_{a,b}$ and have a similar decomposition like (\ref{equation60000}). Hence we only have to consider the High $\times$ Low interaction for  $\tilde{Q}_{a,b}$.
\[
S_{\lambda} \tilde{Q}_{a,b} (S_{\lambda} u, S_{\mu} v) = S_{\lambda, \bullet \geq c_{3}\mu}^{a} \tilde{Q}_{a,b} (S^{a}_{\lambda, \bullet \geq c_{2}\mu} u, S_{\mu} v) + S_{\lambda, \bullet  < c_{1}\mu}^{a} \tilde{Q}_{a,b}(S^{a}_{\lambda, \bullet < c_{2}\mu} u,  S_{\mu} v)
\]
\begin{equation}\label{equation71002}
= S_{\lambda, \bullet \geq c_{3}\mu}^{a} \tilde{Q}_{a,b} (S_{\lambda, \bullet \geq c_{2}\mu} u, S_{\mu} v) +  S_{\lambda, \bullet \geq c \lambda }^{a} \tilde{Q}_{a,b} (S_{\lambda, \bullet < c_{2}\mu}^{-a} u, S_{\mu} v)
+
S_{\lambda, \bullet < c_{1}\mu}^{a} \tilde{Q}_{a,b}(S^{a}_{\lambda, \bullet  < c_{2}\mu} u,  S_{\mu} v).
\end{equation}

We  can do further decomposition for terms $\tilde{B}_{1} $and $\tilde{B}_{2}$ as following:
\[
\tilde{B}_{1} \,= \scriptsize \sum_{\begin{array}{c}
d,\delta_1,\delta_2\\
d < c_{1} \mu, \delta_2 < c_{2}\mu\\

                                                                                                           \end{array}
}\normalsize S_{\lambda, d}^{-b}  Q_{a, b} (S_{\mu, \delta_1}^{a} u, S_{\lambda, \delta_2}^{b} v )  = \widetilde{B_{\textrm{I}}^{1}} + \widetilde{B_{\textrm{II}}^{1}}+\widetilde{B_{\textrm{III}}^{1}}=\sum_{d\leq \mu} S_{\lambda, \bullet < d}^{-b}  Q_{a, b} (S_{\mu, d}^{a} u\,,\,S_{\lambda, \bullet < d}^{b} v)
\]
\begin{equation}\label{equation1105}
+\sum_{d < c_{1} \mu} S_{\lambda, d}^{-b}  Q_{a,b} ( S_{\mu, \bullet \leq \min\{d,\mu\}}^{a} u, S_{\lambda, \bullet \leq \min\{d,c_{2}\mu\}}^b v) + \sum_{d < c_{2} \mu} S_{\lambda, \bullet < d}^{-b}  Q_{a, b} (S_{\mu, \bullet \leq \min\{d,\mu\}}^{a} u , S_{\lambda, d}^{b} v).
\end{equation}

\[
\tilde{B}_{2} = \scriptsize \sum_{\begin{array}{c}
d,\delta_1,\delta_2\\
d < c_{1} \mu, \delta_2 < c_{2}\mu\\

                                                                                                           \end{array}
}\normalsize S_{\lambda, d}^{a} \tilde{Q}_{a, b} (S_{\mu, \delta_1} u, S_{\lambda, \delta_2}^{b} v ) 
 = \widetilde{B_{\textrm{I}}^{2}} + \widetilde{B_{\textrm{II}}^{2}}+\widetilde{B_{\textrm{III}}^{2}}= \sum_{d\leq \mu} S_{\lambda, \bullet < d}^{a} \tilde{Q}_{a, b} (S_{\mu, d} u\,,\,S_{\lambda, \bullet < d}^{b} v)\]
\begin{equation}\label{equation1105}
+\sum_{d < c_{1} \mu} S_{\lambda, d}^{a} \tilde{Q}_{a,b} ( S_{\mu, \bullet \leq \min\{d,\mu\}} u, S_{\lambda, \bullet \leq \min\{d,c_{2}\mu\}}^b v)\,+\sum_{d < c_{2} \mu} S_{\lambda, \bullet < d}^{a} \tilde{Q}_{a, b} (S_{\mu, \bullet \leq \min\{d,\mu\}} u ,S_{\lambda, d}^{b} v).
\end{equation}

Follow a similar argument in \cite{Sterbenz1,Sterbenz2} , we have the following wide angle decomposition lemma for terms $\widetilde{B_{\textrm{I}}^{i}}$
,$\widetilde{B_{\textrm{II}}^{i}}$,$\widetilde{B_{\textrm{III}}^{i}}$, $i \in\{1,2\}$:
\begin{lemma}[Low $\times$ High wide angle decomposition ]\label{lowhigh}
For the bilinear forms
\[
S_{\lambda, \bullet < d}^{-b}  Q_{a, b} (S_{\mu, d}^{a} u\,,\,S_{\lambda, \bullet < d}^{b} v), \,\, S_{\lambda, \bullet < d}^{a} \tilde{Q}_{a, b} (S_{\mu, d} u, S_{\lambda, \bullet < d}^{b} v),
\]
we have following angular decompositions
\[
S_{\lambda, \bullet < d}^{-b} Q_{a, b} (S_{\mu, d}^{a} u\,,\,S_{\lambda, \bullet < d}^{b} v)\]
\[
= \sum_{\tiny \begin{array}{c}
\omega_{1}, \omega_{2},\omega_{3}\\
|\omega_{1} + \omega_{3}| \sim (\frac{d}{\mu})^{1/2}\\
|\omega_{2} - ab \cdot \omega_{3}| \sim (\frac{d}{\mu})^{1/2}\\
\end{array}
}\normalsize B^{\omega_{1}}_{(\frac{d}{\mu})^{1/2}} S^{-b}_{\lambda, \bullet < d} Q_{a, b} (S^{\omega_{2}, a}_{\mu, d} u, B^{\omega_{3}}_{(\frac{d}{\mu})^{1/2}} S^{b}_{\lambda, \bullet < d} v),
\]

\[
S_{\lambda, \bullet < d}^{a} \tilde{Q}_{a, b} (S_{\mu, d}^{c} u\,,\,S_{\lambda, \bullet < d}^{b} v)=\]
\[
 \sum_{\tiny \begin{array}{c}
\omega_{1}, \omega_{2},\omega_{3}\\
|\omega_{1} -  ab\cdot \omega_{3}| \sim (\frac{d}{\mu})^{1/2}\\
|\omega_{2} - bc \cdot \omega_{3}| \sim (\frac{d}{\mu})^{1/2}\\
\end{array}
}\normalsize B^{\omega_{1}}_{(\frac{d}{\mu})^{1/2}} S^{a}_{\lambda, \bullet < d} \tilde{Q}_{a, b} (S^{\omega_{2},c}_{\mu, d} u, B^{\omega_{3}}_{(\frac{d}{\mu})^{1/2}} S^{b}_{\lambda, \bullet < d} v),\quad c\in\{+,-\}.
\]
For the bilinear forms $\widetilde{B_{II}^{i}}$ and $\widetilde{B_{III}^{i}}$, $i\in\{1,2\}$, we have very similar decomposition formulas.
\end{lemma}

\par While for the High $\times$ Low interaction case, due to symmetry, a very similar decomposition lemma holds, to see this point, we can switch the role of $\omega_{2}$ and $\omega_{3}$ and the role of $b$ and $c$, in the Low $\times$ High wide angle decomposition lemma \ref{lowhigh}. Thus for  the High $\times$ Low interaction associated with $\tilde{Q}_{a,b}$, the angle between $\omega_{1}$ and $\omega_{3}$ would be the angle between $\omega_{1}$ and $\omega_{2}$ in lemma \ref{lowhigh} ( remember that the role of $b$ and $c$ are also switched), after switching we can see that 
$
|\omega_{1} - ab\cdot\omega_{3}| \sim ({d}/{\mu})^{1/2}
$ still holds. 
Switching the role of $\omega_{2}$ and $\omega_{3}$ doesn't change the 
angle between $\omega_{2}$ and $\omega_{3}$, thus $|\omega_{2}- ab\cdot \omega_{3}| \sim ({d}/{\mu})^{1/2}$ remains true in the High $\times$ Low interaction case.
\begin{remark}
From above, we can see that for the angular decompositions associated
with $Q_{a,b}$, $|\omega_{2} - ab
\cdot \omega_{3}| \sim ({d}/{\mu})^{1/2}$ always holds for all interaction cases, as for two sector localized inputs, the size of symbol $Q_{a,b}$ is $\angle (\omega_{1}, ab\cdot \omega_{3})$, i.e We can gain $ ({d}/{\mu})^{1/2}$ from $Q_{a,b}$ in the angular decomposition. Very similarly, we can verify that we can also gain $ ({d}/{\mu})^{1/2}$ from $\tilde{Q}_{a,b}$, as for all possible cases, $|\omega_{1}- ab\cdot \omega_{3}| \sim ({d}/{\mu})^{1/2}$ holds in the angular decomposition  and the symbol of $\tilde{Q}_{a,b}$ has size $\angle
(\omega_{1}, ab \omega_{3})$.

\end{remark}

Notice that above decomposition are all based on the fact that $\mu\leq c \lambda$.  When $\mu$ and $\lambda$ have similar size, i.e $c \lambda \leq \mu \lesssim \lambda$. In this case, we know that $\exists C >0$ s.t $S_{\lambda} f = S_{\lambda, \bullet \leq C \mu}^{a} f $ for any $a\in\{+,-\} $. Therefore, we can view term $S_{\mu}(T_{a, b} (S_{\lambda} u, S_{\lambda} v)) $, $T_{a,b}\in\{Q_{a, b}, \tilde{Q}_{a,b}\}$ as $S_{\mu}(T_{a,b} (S_{\lambda, \bullet \leq C\mu}^{c_{1}} u, S_{\lambda, \bullet \leq C \mu}^{c_{2}} v ))$, for suitable choices of $c_{1}, c_{2}\in\{+,-\}$. We can see that it's of  $B_{i}$ ($\tilde{B}_{i}$ for the Low $\times$High interaction) $i\in\{1,2\}$ type terms, hence we can apply the angular decomposition lemma \ref{highhighangular} and \ref{lowhigh}. And it's the only type of terms that we have to deal with when $c\lambda\leq \mu \lesssim \lambda.$

\section{Proof of the main theorem}\label{proof}

In this section, we will give the proof of the main theorems: theorem \ref{maintheorem} and theorem \ref{scattering}. To prove theorem \ref{maintheorem}, it would be sufficient to prove the result for initial data lies in
$(\dot{B}_{2,1}^{0,1},\dot{B}^{\frac{1}{2},1}_{2,1}, \dot{B}^{-\frac{1}{2},1}_{2,1})$, higher regularity result would be derived automatically once we proved all types of bilinear estimates appeared in the critical case.
While for the scattering theorem \ref{scattering}, once our contraction argument works, we would know that $\widetilde{F}_{\Omega}$ norm of $\psi$ and $F^{1/2}_{\Omega}$  norm  of  $\phi$
are bounded, and we could decompose the space-time localized solution into three parts: homogeneous wave solution part, $\langle \Omega\rangle^{-1}X^{\frac{1}{2},1}_{\lambda} \cap  Z_{\Omega, \lambda}$ part and $\langle\Omega\rangle^{-1} Y_{\lambda}$ part. For homogeneous wave part, it's automatically scattering in time; 
for $\langle \Omega\rangle^{-1}X^{\frac{1}{2},1}_{\lambda} \cap  Z_{\Omega, \lambda}$ part, it will approaching to zero as time $t\rightarrow \pm\infty$; for $\langle\Omega\rangle^{-1} Y_{\lambda}$  part, as it's $\langle\Omega\rangle^{-1} Y_{\lambda}$  norm 
is bounded thus  we could construct $(\psi_{0}^{\pm},\phi_{0}^{\pm},\phi_{1}^{\pm})$ explicitly via the Duhamel formula. Detail proof about above argument  could be found in \cite{Sterbenz1,Sterbenz2}.

From now on, we  can restrict ourself to the critical case.   Let's first reduce the proof of the theorem \ref{maintheorem} into the desired estimates.  To apply the Picard iteration method, it would be sufficient to prove the following bilinear 
estimates for functions $u\in F^{1/2}_{\Omega}$ and $v,v_{1}, v_{2} \in \widetilde{F}_{\Omega}$:
\begin{equation}\label{desired1}
\| V_{\pm}\mathcal{N}_{\pm} (u, v) \|_{\widetilde{F}_{\Omega}} \lesssim \| u \|_{F^{1/2}_{\Omega}} \, \| v\|_{\widetilde{F}_{\Omega}},
\end{equation}
\begin{equation}\label{desired2}
\| V \mathcal{N}(v_{1}, v_{2}) \|_{F^{1/2}_{\Omega}}\,\lesssim \| v_{1} \|_{\widetilde{F}_{\Omega}} \, \|v_{2}\|_{\widetilde{F}_{\Omega}}.
\end{equation}
\par After applying Littlewood-Paley decomposition to the nonlinearities $\mathcal{N}_{\pm}(u, v)$ and $\mathcal{N}(v_{1}\, , v_{2})$
 with respect to space-time frequency as following:
\begin{displaymath}
\mathcal{N}_{\pm}(u, v) =\sum_{b\in\{+,-\}} \sum_{\mu, \lambda\in 2^{\mathbb{Z}}} \tilde{Q}_{\pm,b} (S_{\mu} u\,, S_{\lambda} v ),\mathcal{N} (v_{1}, v_{2}) = \sum_{a,b\in\{+,-\}}\sum_{\mu, \lambda\in 2^{\mathbb{Z}}} Q_{a,b} ( S_{\mu} v_{1}, S_{\lambda} v_{2} ),
\end{displaymath}
after 
 separating into Low $\times$ High, High $\times$ Low  and High $\times$ High interaction cases, it would be sufficient to prove the following estimates for each fixed $a, b \in\{+,-\}$, $\mu \lesssim \lambda$:

\begin{equation}\label{desired1}
\| S_{\lambda} V  Q_{a,b} ( S_{\mu} v_{1}, S_{\lambda } v_{2}) \|_{F^{1/2}_{\Omega}} \lesssim \| S_{\mu} v_{1} \|_{\widetilde{F}_{\Omega}} \|S_{\lambda} v_{2} \|_{\widetilde{F}_{\Omega}}, \| S_{\lambda} V_{a}  \tilde{Q}_{a,b} ( S_{\mu} u, S_{\lambda } v) \|_{\widetilde{F}_{\Omega}} \lesssim \|S_{\mu} u \|_{F^{1/2}_{\Omega}} \| S_{\lambda} v \|_{\widetilde{F}_{\Omega}}
\end{equation}

\begin{equation}\label{desired2}
\| S_{\lambda} V Q_{a,b} ( S_{\lambda} v_{1}, S_{\mu } v_{2}) \|_{F^{1/2}_{\Omega}} \lesssim \| S_{\lambda}v_{1} \|_{\widetilde{F}_{\Omega}} \| S_{\mu} v_{2} \|_{\widetilde{F}_{\Omega}}, \| S_{\lambda} V_{a}  \tilde{Q}_{a,b} ( S_{\lambda} u, S_{\mu } v) \|_{\widetilde{F}_{\Omega}} \lesssim \| S_{\lambda}u \|_{F^{1/2}_{\Omega}} \|S_{\mu} v \|_{\widetilde{F}_{\Omega}},
\end{equation}

\begin{equation}\label{desired3}
\sum_{\mu \lesssim \lambda} \,\, \| S_{\mu} V Q_{a,b} (S_{\lambda} v_{1} \,,\, S_{\lambda} v_{2}) \|_{F^{1/2}_{\Omega}} \lesssim \|S_{\lambda} v_{1} \|_{\widetilde{F}_{\Omega}} \,\| S_{\lambda} v_{2} \|_{\widetilde{F}_{\Omega}},
\end{equation}

\begin{equation}\label{desired4}
\sum_{\mu \lesssim \lambda} \,\, \| S_{\mu} V_{a} \tilde{Q}_{a,b} (S_{\lambda} u \,,\, S_{\lambda} v) \|_{\widetilde{F}_{\Omega}} \lesssim \|S_{\lambda}u \|_{{F}_{\Omega}^{1/2}} \,\| S_{\lambda}v \|_{\widetilde{F}_{\Omega}}.
\end{equation}

Let's first prove the Low $\times$ High and High $\times$ Low interaction cases, i.e prove (\ref{desired1}) and (\ref{desired2}).  In this Low $\times$ High and High $\times$ Low interaction case, roughly, we have $S_{\lambda} V T_{a,b}( S_{\mu} f, S_{\lambda} g) \approx V S_{\lambda} T_{a,b}(S_{\mu} f, S_{\lambda} g)$, $T_{a, b}\in\{Q_{a,b}, \tilde{Q}_{a,b}\}.$  $(i):$When $ c \lambda \leq \mu \lesssim \lambda$, we know that 
$S_{\lambda}T_{a,b}(S_{\mu}f, S_{\lambda} g) $ is of $\tilde{B}^{i}$, $i\in\{1,2\}$ type, hence the desired estimate (\ref{desired1}) and (\ref{desired2}) would follow from the result of lemma \ref{lowhighnear}. $(ii)$ When $ \mu \leq c \lambda$, for the bilinear estimate associated with operator ${Q}_{a,b}$, recall the decomposition (\ref{equation60000}) for the Low $\times$ High interaction that we did in the section \ref{Bilinear} and from the estimates (\ref{equation31013}) in lemma \ref{lowhighfar}, (\ref{equation41000}) in lemma \ref{far} and 
(\ref{equation55007})
in lemma \ref{lowhighnear}, we know that the desired Low $\times$ High bilinear estimate associated with ${Q}_{a,b}$ holds and High $\times$ Low interaction for $Q_{a,b}$ follows from symmetry. $(ii:1)$ While for the Low $\times$ High interaction associated with operator $\tilde{Q}_{a,b}$. If $a\cdot b < 0$, recall the decomposition (\ref{equation30001}) and the estimates (\ref{equation40004}) and (\ref{equation40005}) of lemma \ref{far}, we see the desired estimate holds for this case; if $a\cdot b > 0$, recall the decomposition (\ref{equation71001})  and the estimates (\ref{equation31000}) in lemma \ref{lowhighfar} and  (\ref{equation21000}) in lemma \ref{lowhighnear}, we see that the desired estimate also holds in this case. $(ii:2)$ For the High $\times$Low interaction associated with operator $\tilde{Q}_{a,b}$, recall the decomposition (\ref{equation71002}) and the estimates (\ref{equation31003}) in lemma \ref{lowhighfar}, (\ref{equation70000}) in lemma \ref{far} and (\ref{equation55009}) in lemma \ref{lowhighnear}, we could see that the desired estimate also holds. To sum up, (\ref{desired1}) and (\ref{desired2}) also holds when $\mu \leq  c\lambda.$

\vspace{1\baselineskip}
Now let's proceed to the High$\times $ High interaction case, 
due to the fact that operator $V$ is not commute with $S_{\mu}$ in the High $\times$High interaction, as the resulting part of space-time frequencies lies  in the cone would be different. Recall the formula \cite{Sterbenz2}[equation 105],
we can derive the following formula for any two well defined $v_{1}, v_{2}$ and $T_{a,b}\in \{Q_{a,b}, \tilde{Q}_{a,b}\}$:
\begin{equation}\label{equation19}
S_{\mu} V T_{a, b}(S_{1} v_{1}, S_{1} v_{2}) = V S_{\mu} T_{a,b} (S_{1} v_{1}, S_{1} v_{2}) - \sum_{\sigma : \mu < \sigma \leq 1} W(P_{\mu} S_{\sigma, \sigma} V\,T_{a,b} (S_{1} v_{1}, S_{1} v_{2}) ).
\end{equation}
For the commutator terms: the second term of the right hand side of (\ref{equation19}), it could be estimate by the results of lemma \ref{commutator}. Thus the reminded terms to be estimates are the first term of the right hand side of (\ref{equation19}). $(i)$ When $\mu \leq  c \lambda$, recall the decompositions (\ref{equation30005}), (\ref{equation71004}) and (\ref{equation71000}) we did in section \ref{Bilinear}, hence from the estimates of lemma \ref{highhighfar}, lemma \ref{far} and lemma  \ref{highhighnear}, we could see that the desired estimates (\ref{desired3}) and (\ref{desired4}) holds for all possible sign of $a$ and $b$. $(ii)$ When $c \lambda \leq \mu \lesssim \lambda$, we have $S_{\mu} T_{a,b}(S_{\lambda} f, S_{\lambda} g) = S_{\mu} T_{a,b}(S_{\lambda, \bullet \leq C \mu}^{c} f , S_{\lambda, \bullet \leq C \mu}^{d} g )$, $T_{a,b}\in \{Q_{a,b}, \tilde{Q}_{a,b}\}$ $c$ and $d$ are suitable signs depend on $T_{a,b}$. Hence, we can apply the result of lemma \ref{highhighnear}, the it's easy to see the desired estimates (\ref{desired3}) and (\ref{desired4}) also hold.
\subsection{Input or output frequency far away from cone}

\begin{lemma}[High$\times$ High far away from cone]\label{highhighfar}
For $\psi_{1}, \psi_{2}\in \widetilde{F}_{\Omega}$ and $\phi\in F^{1/2}_{\Omega}$, $c_{1} \geq c_{2}+ 4  \geq 6$, we have the 
following estimates:
\begin{equation}\label{equation20006}
\| V S_{\mu}  Q_{a,b} (S_{\lambda,\bullet \geq c_{1}\mu}^{a} \psi_{1}, S^{b}_{\lambda, \bullet \geq c_{2} \mu} \psi_{2})  \|_{F^{1/2}_{\Omega}} 
\lesssim \mu^{\frac{1}{4}-} \| S_{\lambda}\psi_{1}\|_{\widetilde{F}_{\Omega}} \| S_{\lambda} \psi_{2}\|_{\widetilde{F}_{\Omega}},
\end{equation}

\begin{equation}\label{equation20007}
\|  V_{a} S_{\mu} \tilde{Q}_{a,b} (S_{\lambda, \bullet \geq c_{1}\mu} \phi, S^{b}_{\lambda, \bullet \geq c_{2} \mu} \psi_{1})  \|_{\tilde{F}_{\Omega}}
\lesssim \mu^{\frac{1}{4}-} \|  S_{\lambda} \phi\|_{F^{1/2}_{\Omega}} \| S_{\lambda} \psi_{1} \|_{\widetilde{F}_{\Omega}}.
\end{equation}
\end{lemma}

\begin{proof}
Let's first prove (\ref{equation20006}), the proof of (\ref{equation20007}) will be very similar.  As the space we are working on are scale invariant, it would be sufficient to prove (\ref{equation20006}) under the assumption $\lambda =1$. Notice that in the High $\times$ High interaction associated with $Q_{a,b}$ and when $\mu\lesssim c$, $c$ is a sufficient small constant, then $a$ and $b$ has to have same sign (check (\ref{equation3000})). For $(\tau_{1}, \xi_{1}) \in \emph{supp} (\widehat{S_{1, \bullet \lesssim c }^{a} \psi_{1}})$ and $(\tau_{2}, \xi_{2}) \in \emph{supp} (\widehat{S_{1, \bullet \lesssim c}^{b} \psi_{2}})$, we have 
 \[
||\tau_{i}| - |\xi_{i}|| \lesssim  c,\, ||\tau_{i}| + |\xi_{i}|| \sim 1, \, |\tau_{1} - \tau_{2}| + |\xi_{1} - \xi_{2}| \lesssim \mu,\, i \in \{1,2\},
 \]
 hence
 \[
 |\xi_{1}| \sim 1, |\xi_{2}| \sim 1, |\xi_{1} - \xi_{2}| \lesssim \mu \Longrightarrow \angle(\xi_{1}, \xi_{2}) \lesssim \mu.
 \]
Therefore, when the modulations of inputs are both close to the cone, if $a\cdot b > 0$, we can gain extra $\mu$ from the null structure. While for the case $a\cdot b < 0$, then $\mu$ can't be sufficiently small, $\mu \sim 1$.  Another observation is that as the distance between two space-time frequencies is less than $\mu$, then the difference of two modulations can't be too large and is controlled by $2\mu$. Hence, when $\mu \leq c $, we have
\[
 \| V S_{\mu}  Q_{a,b} (S_{1,\bullet \geq c_{1}\mu}^{a} \psi_{1}, S^{b}_{1, \bullet \geq c_{2} \mu} \psi_{2}) \|_{F^{1/2}_{\Omega}} \lesssim \mu^{-\frac{1}{2}} 
\|S_{\mu} Q_{a,b} (S_{1,\bullet \geq c_{1}\mu}^{a}\langle \Omega \rangle \psi_{1}, S^{b}_{1, \bullet \geq c_{2} \mu} \psi_{2})\|_{L^{1}_{t}L^{2}_{x}} + \mu^{-\frac{1}{2}} \cdot \]

\[
\|S_{\mu} Q_{a,b} (S_{1,\bullet \geq c_{1}\mu}^{a} \psi_{1}, S^{b}_{1, \bullet \geq c_{2} \mu} \langle \Omega \rangle \psi_{2})\|_{L^{1}_{t}L^{2}_{x}}  \lesssim \mu^{-\frac{1}{2}}
 \|S_{\mu} Q_{a,b} (S_{1, 2 c \geq \bullet \geq c_{1}\mu}^{a}  \langle \Omega \rangle\psi_{1}, S^{b}_{1, 4 c \geq \bullet \geq c_{2} \mu} \psi_{2}) \|_{L^{1}_{t}L^{2}_{x}} 
\]

\[
+\mu^{-\frac{1}{2}}  \|S_{\mu} \big[ Q_{a,b} (S_{1,  \bullet \geq 2 c}^{a} \langle \Omega \rangle \psi_{1}, S^{b}_{1,   \bullet \geq c} \psi_{2}) \big]\|_{L^{1}_{t}L^{2}_{x}} +\mu^{-\frac{1}{2}}  \|S_{\mu} \big[ Q_{a,b} (S_{1,  \bullet \geq 2 c}^{a}  \psi_{1}, S^{b}_{1,   \bullet \geq c} \langle \Omega \rangle \psi_{2}) \big]\|_{L^{1}_{t}L^{2}_{x}} 
\]

\[
+ \mu^{-\frac{1}{2}}
 \|S_{\mu} Q_{a,b} (S_{1, 2 c \geq \bullet \geq c_{1}\mu}^{a}  \psi_{1}, S^{b}_{1, 4 c \geq \bullet \geq c_{2} \mu} \langle \Omega \rangle\psi_{2}) \|_{L^{1}_{t}L^{2}_{x}}
\]
\[
 \lesssim  \mu^{\frac{5}{4}-}\big(\sum_{4 c_{}\geq d \geq c_{2}\mu} \| S^{a}_{1,d}\langle \Omega \rangle \psi_{1}\|_{L^{2}_{t}L^{2}_{x}} \|S^{b}_{1, 4 c_{} \geq \bullet \geq c_{2} \mu} \psi_{2}\|_{L^{2}_{t}L^{4+}_{x}} + \| S_{1,d}^{b}\langle \Omega \rangle \psi_{2}\|_{L^{2}_{t}L^{2}_{x}} \| S_{1, 2 c_{} \geq \bullet \geq c_{1}\mu}^{a}  \psi_{1}\|_{L^{2}_{t}L^{4+}_{x}}\big)
  \]

\[
+ \mu^{\frac{1}{4}-} \big(\sum_{d \geq c_{}} \| S^{a}_{1,d}\langle \Omega \rangle \psi_{1}\|_{L^{2}_{t}L^{2}_{x}} \|S^{b}_{1,  \bullet \geq c_{}} \psi_{2}\|_{L^{2}_{t}L^{4+}_{x}} + \| S_{1,d}^{b}\langle \Omega \rangle \psi_{2}\|_{L^{2}_{t}L^{2}_{x}} \| S_{1,  \bullet \geq  2c_{}}^{a}  \psi_{1}\|_{L^{2}_{t}L^{4+}_{x}}\big)
\]
\[
\lesssim  \mu^{\frac{5}{4}-} \sum_{d \geq c_{2}\mu} d^{-\frac{1}{2}} \| S_{1} \psi_{1}\|_{\widetilde{F}_{\Omega}} \| S_{1}\psi_{2}\|_{\widetilde{F}_{\Omega}} + \mu^{\frac{1}{4}-} \sum_{d \geq c_{}} d^{-\frac{1}{2}} \| S_{1} \psi_{1}\|_{\widetilde{F}_{\Omega}} \| S_{1} \psi_{2}\|_{\widetilde{F}_{\Omega}}
\]

\begin{equation}\label{equation20008}
 \lesssim \mu^{\frac{1}{4}-} \|S_{1} \psi_{1}\|_{\widetilde{F}_{\Omega}} \|S_{1} \psi_{2}\|_{\widetilde{F}_{\Omega}}.
\end{equation}
When $c\leq \mu \lesssim 1$, we don't need to utilize the null structure to gain $\mu$ and follow the same method used in the proof of (\ref{equation20008}), we can derive the upper bound $\mu^{-({1}/{4}+)}\| \psi_{1}\|_{\widetilde{F}_{\Omega}} \| \psi_{2}\|_{\widetilde{F}_{\Omega}}$, as $\mu\sim 1$, hence (\ref{equation20006}) also holds.

\par The proof of (\ref{equation20007}) is very similar, however, as th null structure of $\tilde{Q}_{a,b}$ is the angle between one of the inputs and the output, generally we don't expect to gain smallness in the High $\times$ High interaction, the angle could be large. While very luckily, in (\ref{equation20007}), we won't loss extra $\mu^{1/2}$ from the regularity, the total loss from modulation will be $\mu^{1/2}$ and we can gain $\mu^{3-/4}$ from the Sobolev embedding (See lemma \ref{Sobolev}) after utilizing the improved Strichartz estimate (see (\ref{equation30})). Total gain dominate total loss, hence (\ref{equation20007}) holds.

\end{proof}

\begin{lemma}[Low$\times$ High and High $\times$ Low far away from cone]\label{lowhighfar}
For $\psi_{1}, \psi_{2}\in \widetilde{F}_{\Omega}$ and $\phi\in F^{1/2}_{\Omega}$, $c_{1} \geq c_{2}+ 4 \geq c_{3}+ 8 \geq 12$, we have the 
following estimates:

\begin{equation}\label{equation31013}
\| V   S_{\lambda, \bullet \geq c_{3}\mu} Q_{a,b} (S_{\mu} \psi_{1}, S_{\lambda, \bullet \geq c_{2} \mu}^{b} \psi_{2})\|_{F^{1/2}_{\Omega}}
 \lesssim
\| S_{\mu}\psi_{1}\|_{\widetilde{F}_{\Omega}} \| S_{\lambda} \psi_{2}\|_{\widetilde{F}_{\Omega}},
\end{equation}

\begin{equation}\label{equation31003}
\| V_{a}   S_{\lambda, \bullet \geq c_{3}\mu}^{a} \tilde{Q}_{a,b} (S_{\lambda, \bullet \geq c_{2} \mu} \phi, S_{\mu} \psi_{1}) \big]\|_{\widetilde{F}_{\Omega}}
 \lesssim  
\| S_{\lambda}\phi\|_{F^{1/2}_{\Omega}} \| S_{\mu} \psi_{1}\|_{\widetilde{F}_{\Omega}},
\end{equation}

\begin{equation}\label{equation31000}
\| V_{a}   S_{\lambda, \bullet \geq c_{3}\mu}^{a} \tilde{Q}_{a,b} (S_{\mu} \phi, S_{\lambda, \bullet \geq c_{2} \mu}^{b} \psi_{1}) \big]\|_{\widetilde{F}_{\Omega}}
 \lesssim
\| S_{\mu}\phi\|_{F^{1/2}_{\Omega}} \| S_{\lambda} \psi_{1}\|_{\widetilde{F}_{\Omega}}, \quad \emph{here $a\cdot b > 0$ }.
\end{equation}

\end{lemma}

\begin{proof}
As the function spaces are scale invariant, it would be sufficient to prove above estimates under assumption $\lambda =1$. Hence
\[
\| V   S_{1, \bullet \geq c_{3}\mu} Q_{a,b} (S_{\mu} \psi_{1}, S_{1,  \bullet \geq c_{2} \mu}^{b} \psi_{2})\|_{F^{1/2}_{\Omega}}  \lesssim \sum_{d \geq c_{3} {\mu}} d^{-\frac{1}{2}} \| S_{1,d} Q_{a,b} (S_{\mu} \langle \Omega \rangle \psi_{1}, S_{1, \bullet \geq c_{2}\mu}^{b} \psi_{2})\|_{L^{2}_{t} L^{2}_{x}} + 
\]
\[
 d^{-\frac{1}{2}}\| S_{1,d} Q_{a,b} (S_{\mu} \psi_{1}, S_{1,   \bullet \geq c_{2}\mu}^{b} \langle \Omega \rangle \psi_{2})\|_{L^{2}_{t}L^{2}_{x}} 
\lesssim \sum_{d \geq c_{3} {\mu}} d^{-\frac{1}{2}}
\| S_{\mu}\langle \Omega\rangle \psi_{1}\|_{L^{\infty}_{t}L^{4-}_{x}} \| S_{1, \bullet \geq c_{2}\mu}^{b} \psi_{2}\|_{L^{2}_{t}L^{4+}_{x}}\]

\[
+ d^{-\frac{1}{2}} \| S_{\mu}\psi_{1}\|_{L^{2}_{t}L^{\infty}_{x}}  \| S_{1,   \bullet \geq c_{1}\mu}^{b} \langle \Omega \rangle \psi_{2}\|_{L^{\infty}_{t}L^{2}_{x}} \lesssim \sum_{ d \geq c_{3}\mu} d^{-\frac{1}{2}} \mu^{\frac{3}{4}-} \| S_{\mu}\psi_{1}\|_{\widetilde{F}_{\Omega}} \| S_{1} \psi_{2}\|_{\widetilde{F}_{\Omega}}  + 
\]
\begin{equation}\label{equation31011}
\sum_{d\geq c_{3}\mu}d^{-\frac{1}{2}}\mu \| S_{\mu}\psi_{1}\|_{\widetilde{F}_{\Omega}} \| S_{1}\psi_{2}\|_{\widetilde{F}_{\Omega}} \lesssim \mu^{\frac{1}{4}-}\| S_{\mu}\psi_{1}\|_{\widetilde{F}_{\Omega}} \| S_{1}\psi_{2}\|_{\widetilde{F}_{\Omega}}.
\end{equation}

The proof of (\ref{equation31003}) is very similar, we omit the detail here. However
to prove (\ref{equation31000}), if one wants to use above rough estimate, the most we can gain is $\mu^{1/2-}$, while we will lose $\mu^{1/2}$ from the regularity of $\phi$. The observation is that we can gain extra $\mu$ from the null structure when both inputs are close to the cone.
 As we are in the Low $\times$ High interaction case, the angle between the output and the high input should be small when both of them are close to the cone. More precisely, suppose that $(\tau_{1}, \xi_{1}) \in \emph{supp} (\widehat{S_{\mu}\phi}) $, $(\tau_{2}, \xi_{2}) \in \emph{supp} (\mathcal{F}_{t,x}(S_{1, c \geq \bullet \geq c \mu}^{b} \psi_{1}))$, then we have
\[
|\tau_{1}| + |\xi_{1}| \sim \mu, \quad |\tau_{2}| + |\xi_{2}| \sim 1, \quad c\mu \leq ||\tau_{2}| - |\xi_{2}|| \leq c,
 \]
 hence
 \begin{equation}\label{equation21001}
|\xi_{2}| \sim |\xi_{1}+\xi_{2}| \sim 1, |\xi_{1}| \lesssim \mu \Longrightarrow \angle(\xi_{1}+\xi_{2}, \xi_{2}) \lesssim \mu,
 \end{equation}
 from (\ref{equation21001}), we can see that we actually can gain $\mu$ from the null structure of $\tilde{Q}_{a,b}$ as $a\cdot b > 0$, therefore
\begin{equation}
\| V_{a}   S_{1, \bullet \geq c_{3}\mu}^{a} \tilde{Q}_{a,b} (S_{\mu} \phi, S_{1, \bullet \geq c_{2} \mu}^{b} \psi_{1}) \big]\|_{\widetilde{F}_{\Omega}}
 \lesssim 
 \sum_{d \geq c} d^{-\frac{1}{2}} \mu^{\frac{1}{4}-} \| S_{\mu} \phi\|_{F^{1/2}_{\Omega}}\| S_{1} \psi_{1}\|_{\widetilde{F}_{\Omega}}  \end{equation}
\begin{equation}
+ \sum_{c\mu \leq d \leq c} d^{-\frac{1}{2}} \mu^{\frac{3}{4}-} \mu^{\frac{1}{2}}
\| S_{\mu} \phi\|_{F^{1/2}_{\Omega}}\| S_{1} \psi_{1}\|_{\widetilde{F}_{\Omega}}   \lesssim \mu^{\frac{1}{4}-}  \| S_{\mu} \phi\|_{F^{1/2}_{\Omega}}\| S_{1} \psi_{1}\|_{\widetilde{F}_{\Omega}}.
\end{equation}

\end{proof}
\begin{lemma}\label{far}
For $\psi_{1}, \psi_{2},\psi\in \widetilde{F}_{\Omega}$ and $\phi\in F^{1/2}_{\Omega}$, we have the  
following estimates:

\begin{equation}\label{equation40001}
\| V S_{\mu}Q_{a, b}(S^{a}_{\lambda, \bullet \geq c\lambda} \psi_{1}, S^{}_{\lambda} \psi_{2}) \|_{F^{1/2}_{\Omega}} \lesssim \mu^{\frac{3}{p}-\frac{1}{2}} \| S_{\lambda}\psi_{1}\|_{\widetilde{F}_{\Omega}} \| S_{\lambda}\psi_{2}\|_{\widetilde{F}_{\Omega}},
\end{equation}

\begin{equation}\label{equation40002}
\|V S_{\mu} Q_{a, b} (S^{a}_{\lambda, \bullet \leq c \lambda} \psi_{1}, S^{b}_{\lambda, \bullet \geq  c\lambda} \psi_{2}) \|_{F^{1/2}_{\Omega}}
\lesssim \mu^{\frac{3}{p}-\frac{1}{2}} \| S_{\lambda}\psi_{1}\|_{\widetilde{F}_{\Omega}} \| S_{\lambda}\psi_{2}\|_{\widetilde{F}_{\Omega}},
\end{equation}
\begin{equation}\label{equation40003}
\|V S_{\mu} \tilde{Q}_{a, b} (S^{b}_{\lambda, \bullet < c_{2} \mu} \phi, S^{b}_{\lambda, \bullet \geq  c\lambda} \psi) \|_{\widetilde{F}_{\Omega}}
\lesssim \mu^{\frac{3}{p}} \| S_{\lambda}\phi\|_{{F}_{\Omega}^{1/2}} \| S_{\lambda}\psi\|_{\widetilde{F}_{\Omega}},
\end{equation}

\begin{equation}\label{equation40004}
  \|  V_{a} S_{\lambda} \tilde{Q}_{a,b} (S_{\mu} \phi, S_{\lambda, \bullet \geq c \lambda}^{b} \psi)\|_{\widetilde{F}_{\Omega}}
\lesssim \| S_{\mu}\phi\|_{F^{1/2}_{\Omega}} \| S_{\lambda}\psi\|_{\widetilde{F}_{\Omega}},
\end{equation}

\begin{equation}\label{equation41000}
\|V S_{\lambda, \bullet \leq c_{3}\mu}^{b} Q_{a,b} (S_{\mu} \psi_{1}, S_{\lambda, \bullet \geq c \lambda}^{b} \psi_{2})\|_{F^{1/2}_{\Omega}} \lesssim \| S_{\mu}\psi_{1}\|_{\widetilde{F}_{\Omega}} \| S_{\lambda}\psi_{2}\|_{\widetilde{F}_{\Omega}},
\end{equation}

\begin{equation}\label{equation40005}
\|V_{a} S^{a}_{\lambda, \bullet \geq c\lambda} \tilde{Q}_{a, b}(S_{\mu} \phi, S^{b}_{\lambda, \bullet < c \lambda} \psi)\|_{\widetilde{F}_{\Omega}} \lesssim \| S_{\mu} \phi\|_{F^{1/2}_{\Omega}} \| S_{\lambda} \psi\|_{\widetilde{F}_{\Omega}},
\end{equation}
 
 \begin{equation}\label{equation70000}
 \| V_{a} S^{a}_{\lambda, \bullet \geq c\lambda} \tilde{Q}_{a,b}( S^{-a}_{\lambda, \bullet < c_{2}\mu} \phi, S_{\mu}\psi)\|_{\widetilde{F}_{\Omega}} \lesssim  \| S_{\lambda} \phi\|_{F^{1/2}_{\Omega}} \| S_{\mu}\psi \|_{\widetilde{F}_{\Omega}}.
 \end{equation}

\end{lemma}

\begin{proof}

As the function spaces are scale invariant, it would be sufficient to prove above estimates under assumption $\lambda =1$. To prove (\ref{equation40001}), we have 
\[
\|V S_{\mu}Q_{a, b}(S^{a}_{1, \bullet \geq c} \psi_{1}, S^{}_{1} \psi_{2}) \|_{F^{1/2}_{\Omega}} \lesssim \mu^{-\frac{1}{2}+\frac{3}{4}-} \sum_{d \geq c } \| S^{a}_{1,d} \langle \Omega\rangle \psi_{1}\|_{L^{2}_{t}L^{2}_{x}} \| S^{}_{1}\psi_{2}\|_{L^{2}_{t}L^{4+}_{x}} +\]

\[
\mu^{-\frac{1}{2}+\frac{3}{p}} \sum_{ d \geq c} d^{- 1-\frac{2}{p}} \| \sup_{\omega }\| S^{\omega, a}_{1, d}\psi_{1}\|_{L^{p}_{x}}  \|_{L^{1}_{t}}  \| S^{}_{1}\langle \Omega \rangle \psi_{2}\|_{L^{\infty}_{t}L^{2}_{x}} 
\]

\begin{equation}\label{equation40006}
\lesssim \mu^{\frac{3}{p}-\frac{1}{2}} \|S_{1} \psi_{1}\|_{\widetilde{F}_{\Omega}} \| S_{1} \psi_{2}\|_{\widetilde{F}_{\Omega}}.
\end{equation}
The proof of (\ref{equation40002}) and (\ref{equation40003})  are very similar, we omit the detail here. To prove (\ref{equation40004}), notice that
\[
  \|  V_{a} S_{1} \tilde{Q}_{a,b} (S_{\mu} \phi, S_{1, \bullet \geq c }^{b} \psi)\|_{\widetilde{F}_{\Omega}} \lesssim  \mu^{\frac{3}{4}-}\sum_{ d \geq c } \| S_{\mu}\phi\|_{L^{2}_{t}L^{4+}_{x}} \| S^{b}_{1, d}\langle \Omega \rangle \psi\|_{L^{2}_{t}L^{2}_{x}} +
\]
\[
\sum_{ d\geq c} \| S_{\mu} \langle \Omega \rangle\phi\|_{L^{2+}_{t}L^{\infty-}_{x}} \| S^{b}_{1,d}\psi\|_{L^{2-}_{t}L^{2+}_{x}} \lesssim \mu^{\frac{1}{2}-} \|S_{\mu} \phi\|_{F^{1/2}_{\Omega}} \|S_{1} \psi \|_{\widetilde{F}_{\Omega}} 
\]

\[
+ \sum_{d \geq c } \mu^{\frac{1}{2}-} \| S_{\mu}\phi\|_{F^{1/2}_{\Omega}} \| S_{1,d}^{b} \psi\|_{L^{2}_{t}L^{2}_{x}}^{1-} \| S_{1,d}^{b}\psi\|_{L^{1}_{t}L^{\infty}_{x}}^{0+} \lesssim \mu^{\frac{1}{2}-} \| S_{\mu}\phi\|_{F^{1/2}_{\Omega}} \| S_{1}\psi\|_{\widetilde{F}_{\Omega}}.
\]
Hence (\ref{equation40004}) holds and the proof of (\ref{equation41000}) is very similar.
To prove (\ref{equation40005}), notice that
\[
\|V_{a} S^{a}_{1, \bullet \geq c} \tilde{Q}_{a, b}(S_{\mu} \phi, S^{b}_{1, \bullet < c } \psi)\|_{\widetilde{F}_{\Omega}} \lesssim \|  V_{a} S^{a}_{1, \bullet \geq c} \tilde{Q}_{a, b}(S_{\mu} \phi, S^{b}_{1, \bullet < c \lambda} \psi)\|_{\langle \Omega \rangle^{-1}X^{1/2,1}_{1} \cap Z_{\Omega, 1}},
\]
as 
\[
\|V_{a} S^{a}_{1, \bullet \geq c} \tilde{Q}_{a, b}(S_{\mu} \phi, S^{b}_{1, \bullet < c } \psi)\|_{\langle \Omega \rangle^{-1}X^{1/2,1}_{1}} \lesssim \sum_{d \geq c} d^{-\frac{1}{2}} \| S_{\mu} \langle \Omega \rangle\phi\|_{L^{\infty}_{t}L^{4-}_{x}} \| S^{b}_{1,\bullet < c}\psi\|_{L^{2}_{t}L^{4+}_{x}}+
\]

\begin{equation}\label{equation40008}
\sum_{ d \geq c } d^{-\frac{1}{2}} \| S_{\mu}\phi\|_{L^{2}_{t}L^{\infty}_{x}} \| S^{b}_{1, \bullet < c}\langle \Omega\rangle \psi\|_{L^{\infty}_{t}L^{2}_{x}} \lesssim \mu^{\frac{1}{4}-} \| S_{\mu}\phi\|_{F^{1/2}_{\Omega}} \| S_{1}\psi\|_{\widetilde{F}_{\Omega}},
\end{equation}

\[
\|V_{a} S^{a}_{1, \bullet \geq c} \tilde{Q}_{a, b}(S_{\mu} \phi, S^{b}_{1, \bullet < c } \psi)\|_{Z_{\Omega, 1}} \lesssim \sum_{ d \geq c} d^{-1+ \frac{1.1}{p}} \big\| \sup_{\omega} \| S_{1, d}^{\omega, a} \tilde{Q}_{a,b} (S_{\mu}\phi, S^{b}_{1, \bullet < c}\psi)\|_{L^{p}_{x}} \big\|_{L^{1}_{t}}
\]
 
 \begin{equation}\label{equation40007}
 \lesssim  \| S_{\mu}\phi\|_{L^{2}_{t}L^{\infty}_{x}} \| S^{b}_{1, \bullet < c}\psi\|_{L^{2}_{t}L^{4+}_{x}} \lesssim \mu^{\frac{1}{2}} \| S_{\mu} \phi\|_{{F}_{\Omega}^{1/2}} \| S_{1}\psi \|_{\widetilde{F}_{\Omega}}.
 \end{equation}
 Hence from (\ref{equation40008}) and (\ref{equation40007}), we could see that (\ref{equation40005}) holds. The proof of (\ref{equation70000}) is very similar.
\end{proof}

\subsection{High $\times$High: near the cone}
\begin{lemma}\label{highhighnear}
For $\psi_{1}, \psi_{2}\in \widetilde{F}_{\Omega}$ and $\phi\in F^{1/2}_{\Omega}$, $c_{1} \geq c_{2}+ 4 \geq 8$, we have the 
following estimates:
\begin{equation}\label{equation21010}
\| V S_{\mu}  Q_{a,b} (S_{\lambda,\bullet < c_{1}\mu}^{a} \psi_{1}, S^{b}_{\lambda, \bullet < c_{2} \mu} \psi_{2})  \|_{F^{1/2}_{\Omega}} 
\lesssim \mu^{\frac{1}{p}}\| S_{\lambda} \psi_{1} \|_{\widetilde{F}_{\Omega}} \| S_{\lambda} \psi_{2} \|_{\widetilde{F}_{\Omega}},
\end{equation}
\begin{equation}\label{equation21023}
 \|  V_{a} S_{\mu} \tilde{Q}_{a,b} (S_{\lambda, \bullet < c_{1}\mu}^{-b} \phi, S^{b}_{\lambda, \bullet < c_{2} \mu} \psi_{1})  \|_{\widetilde{F}_{\Omega}}
\lesssim \mu^{\frac{1}{2}} \|  S_{\lambda} \phi\|_{F^{1/2}_{\Omega}} \| S_{\lambda} \psi_{1} \|_{\widetilde{F}_{\Omega}}.\end{equation}
\end{lemma}
\begin{proof}
As the function spaces are scale invariant, it would be sufficient to prove above estimates under assumption $\lambda =1$. Let's prove (\ref{equation21010}) first, as
\[
 S_{\mu}  Q_{a,b} (S_{1,\bullet < c_{1}\mu}^{a} \psi_{1}, S^{b}_{1, \bullet < c_{2} \mu} \psi_{2}) = \mathcal{I}_{1} + \mathcal{I}_{2} +\mathcal{I}_{3},
\]
where

\[
\mathcal{I}_{1} \,=\, \sum_{d\leq \mu} {S}_{\mu, d} Q_{a, b} (S_{1, \bullet < d}^{a} \psi_{1} \, ,\,S_{1, \bullet < d}^{b} \psi_{2}),\mathcal{I}_{2} = \sum_{ d < c_{2} \mu} {S}_{\mu, \bullet \leq \min\{d, \mu\}} Q_{a, b} (S_{1, \bullet \leq d}^{a} \psi_{1}\,,\, S_{1, d}^{b} \psi_{2})
\]
\[
\mathcal{I}_{3} = \sum_{ d < c_{1} \mu} {S}_{\mu, \bullet \leq \min\{d, \mu\}} Q_{a, b}( S_{1, d}^{a} \psi_{1}\, ,\, S_{1, \bullet < \min\{d, c_{2}\mu\}}^{b} \psi_{2}).
\]
Recall the wide angle decomposition that we did in the section \ref{Bilinear}, hence we have

\[
\mu^{\frac{1}{2}}\| V \mathcal{I}_{1}\|_{\langle \Omega\rangle^{-1}X^{1/2, 1}_{\mu}} \lesssim \mathcal{I}_{1,1} + \mathcal{I}_{1,2}:= \sum_{d\lesssim \mu}  \mu^{-\frac{1}{2}} d^{-\frac{1}{2}} \|{S}_{\mu, d} Q_{a, b} (S_{\lambda, \bullet < d}^{a} \psi_{1} ,
S_{\lambda, \bullet < d}^{b} \langle \Omega \rangle \psi_{2})\|_{L^{2}_{t} L^{2}_{x}}\]
\[
+  \sum_{d\lesssim \mu} \mu^{-\frac{1}{2}} d^{-\frac{1}{2}} \|{S}_{\mu, d} Q_{a, b} (S_{\lambda, \bullet < d}^{a}  \langle \Omega \rangle  \psi_{1} \, ,\,S_{\lambda, \bullet < d}^{b} \psi_{2})\|_{L^{2}_{t} L^{2}_{x}}
\]

\[
\mathcal{I}_{1,1} \lesssim \sum_{d:\,\,d \lesssim \mu} d^{-\frac{1}{2}} \mu^{-\frac{1}{2}}  \Big\|    \sum_{\scriptsize \begin{array}{c}
|\omega_{1} \mp a\cdot\omega_{2}| \sim (\frac{d}{\mu})^{\frac{1}{2}}\\
|\omega_{2} - ab\cdot \omega_{3}| \sim (\frac{d}{\mu})^{\frac{1}{2}}\\
\end{array}}  S^{\omega_1,\pm}_{\mu, d} Q_{a,b} \big( B^{\omega_2}_{(\frac{d}{\mu})^{\frac{1}{2}}} S^{a}_{1,\bullet<d}\psi_{1}, B^{\omega_3}_{(\frac{d}{\mu})^{\frac{1}{2}}} S^{b}_{1, \bullet < d} \langle \Omega \rangle  \psi_{2}\big)      \Big\|_{L^{2}_{t}L^{2}_{x}}\]

\[
\lesssim   \sum_{ d \lesssim  \mu }(d\mu)^{-\frac{1}{2}} 
\Big\| \Big( \sum_{\omega_1,\omega_2,\omega_3} \big\| S^{\omega_1,\pm}_{\mu, d} Q_{a,b} \big( B^{\omega_2}_{(\frac{d}{\mu})^{\frac{1}{2}}} S^{a}_{1, \bullet < d}  \psi_{1}, 
 B^{\omega_3}_{(\frac{d}{\mu})^{\frac{1}{2}}} S^{b}_{1, \bullet <d} \langle \Omega\rangle \psi_{2}\big)\big\|_{L^{2}_{x}}^{2} \Big)^{\frac{1}{2}}        \Big\|_{L^{2}_{t}}  \]

\[
 \lesssim \sum_{ d \lesssim  \mu} d^{-\frac{1}{2}}\mu^{-\alpha+\frac{5}{4}-}\Big(\frac{d}{\mu}\Big)^{\frac{\alpha}{2}+\frac{1}{4}-} \Big\| \sup_{\omega_{2}} \| B^{\omega_2}_{(\frac{d}{\mu})^{\frac{1}{2}}} S^{a}_{1,  <d} \psi_{1}\|_{L^{4+}_{x}} \big( \sum_{\omega_{3}} \|B^{\omega_3}_{(\frac{d}{\mu})^{\frac{1}{2}}} S^{b}_{1, \bullet < d} \langle \Omega\rangle\psi_{2}\|_{L^{2}_{x}}^{2} \Big)^{\frac{1}{2}}   \Big\|_{L^{2}_{t}}
\]

\[
\lesssim \sum_{ d \lesssim  \mu} 
 d^{-\frac{1}{2}}\mu^{-\alpha+\frac{5}{4}-}\Big(\frac{d}{\mu}\Big)^{\frac{\alpha}{2}+\frac{1}{4}-} \sup_{\omega_{2}} \| B^{\omega_2}_{(\frac{d}{\mu})^{\frac{1}{2}}} S^{a}_{1, \bullet < d} \psi_{1}\|_{L^{2}_{t}L^{4+}_{x}} \|S_{1} \langle \Omega \rangle\psi_{2}\|_{L^{\infty}_{t}L^{2}_{x}}
\]
\begin{equation}\label{equation21012}
\lesssim  \sum_{ d \lesssim  \mu} 
 d^{-\frac{1}{2}}\mu^{-\alpha+\frac{5}{4}-}\Big(\frac{d}{\mu}\Big)^{\frac{\alpha}{2}+\frac{1}{2}-} \| S_{1}\psi_{1}\|_{\widetilde{F}_{\Omega}} \| S_{1}\psi_{2}\|_{\widetilde{F}_{\Omega}} \lesssim \mu^{-\alpha + \frac{3}{4}-}  \| S_{1}\psi_{1}\|_{\widetilde{F}_{\Omega}} \| S_{1}\psi_{2}\|_{\widetilde{F}_{\Omega}}.
\end{equation}
By symmetry, we can switch the role of $\psi_{1}$ and $\psi_{2}$ and  estimate $\mathcal{I}_{1,2}$ in the same way. Hence
we have
\begin{equation}\label{equation51030}
\| V \mathcal{I}_{1}\|_{\langle \Omega\rangle^{-1}X^{1/2, 1}_{\mu}}  \lesssim \mu^{-\alpha+\frac{3}{4}-}  \| S_{1}\psi_{1}\|_{\widetilde{F}_{\Omega}} \| S_{1}\psi_{2}\|_{\widetilde{F}_{\Omega}}.
\end{equation}

\begin{equation}
\mu^{\frac{1}{2}}\| V \mathcal{I}_{1}\|_{Z_{\Omega, \mu}} \lesssim \mu^{-1+\frac{3}{p}} \sum_{d\leq \mu} d^{-1} \Big(\frac{d}{\mu}\Big)^{\frac{1.1}{p}} \int \sup_{\omega} \| S_{\mu, d}^{\omega} \mathcal{I}_{1}\|_{L^{p}_{x}} \, d\, t
\end{equation}

\[
\lesssim  \mu^{-\alpha+\frac{3}{p}} \sum_{d\leq \mu} d^{-1} \Big(\frac{d}{\mu}\Big)^{\frac{\alpha}{2}+\frac{1.1}{p}} \mu^{3(-\frac{1}{p}+\frac{1}{2}-)} \Big(\frac{d}{\mu}\Big)^{(-\frac{1}{p}+\frac{1}{2}-)} \| B^{\omega_2}_{(\frac{d}{\mu})^{\frac{1}{2}}} S^{a}_{1,\bullet< d} \psi_{1}\|_{L^{2}_{t}L^{4+}_{x}}\| B^{\omega_3}_{(\frac{d}{\mu})^{\frac{1}{2}}} S^{b}_{1, \bullet < d} \psi_{2}\|_{L^{2}_{t}L^{4+}_{x}}\]

\begin{equation}\label{equation21025}
 \lesssim \mu^{-\alpha+\frac{3}{p}} \sum_{d\leq \mu} d^{-1} \Big(\frac{d}{\mu}\Big)^{\frac{\alpha}{2}+1} \mu^{3(-\frac{1}{p}+\frac{1}{2}-)} \| S_{1}\psi_{1}\|_{\widetilde{F}_{\Omega}} \| S_{1} \psi_{2}\|_{\widetilde{F}_{\Omega}}\lesssim \mu^{-\alpha+ \frac{1}{2}-} \| S_{1}\psi_{1}\|_{\widetilde{F}_{\Omega}} \| S_{1} \psi_{2}\|_{\widetilde{F}_{\Omega}}.
\end{equation}

\noindent Here, if $a\cdot b >0$, then we choose $\alpha\in(0,1/5)$. In this case, we can utilize the fact that the gained angle is less than both $\mu$ and $({d}/{\mu})^{{1}/{2}}$, hence we could use the upper bound $({d}/{\mu})^{{\alpha}/{2}}\mu^{1-\alpha}$. While if $a\cdot b < 0$, we have to let $\alpha=1$, however in this case $\mu\sim1$. Hence after combining the results of (\ref{equation51030}) and (\ref{equation21025}), we have the following estimate for any pair of sign $a$ and $b$ :
\begin{equation}\label{equation51040}
\| V\mathcal{I}_{1}\|_{F^{1/2}_{\Omega}} \lesssim \mu^{\frac{1}{4}}   \| S_{1}\psi_{1}\|_{\widetilde{F}_{\Omega}} \| S_{1}\psi_{2}\|_{\widetilde{F}_{\Omega}}.
\end{equation}
Let's proceed to estimate term $\mathcal{I}_{2}$, we have

\[
\| V \mathcal{I}_{2}\|_{F^{1/2}_{\Omega}} \lesssim
\mathcal{I}_{2,1} + \mathcal{I}_{2,2} := \sum_{d < c_{2} \mu} \mu^{-\frac{1}{2}} \big\|{S}_{\mu, \bullet \leq \min\{d, \mu\}} Q_{a, b} (S_{1, \bullet \leq d}^{a} \psi_{1}\,,\, S_{1, d}^{b}  \langle \Omega \rangle  \psi_{2})\big\|_{L^{1}_{t}L^{2}_{x}}
\]

\[ +  
\sum_{d < c_{2} \mu} \mu^{-\frac{1}{2}} \big\|{S}_{\mu, \bullet \leq \min\{d, \mu\}} Q_{a, b} (S_{1, \bullet \leq d}^{a} \langle \Omega \rangle \psi_{1}\,,\, S_{1, d}^{b} \psi_{2})\big\|_{L^{1}_{t}L^{2}_{x}} .
\]

\[
\mathcal{I}_{2,1} \lesssim \sum_{d < c_{2} \mu} \mu^{-\frac{1}{2}}  \Big\|    \sum_{\scriptsize \begin{array}{c}
|\omega_{1} \mp a\cdot\omega_{2}| \sim (\frac{d}{\mu})^{\frac{1}{2}}\\
|\omega_{2} - ab\cdot \omega_{3}| \sim (\frac{d}{\mu})^{\frac{1}{2}}\\
\end{array}}  S^{\omega_1,\pm}_{\mu, \bullet \leq \min\{d, \mu\}} Q_{a,b} \big( B^{\omega_2}_{(\frac{d}{\mu})^{\frac{1}{2}}} S^{a}_{1,\bullet \leq d} \psi_{1}, B^{\omega_3}_{(\frac{d}{\mu})^{\frac{1}{2}}} S^{b}_{1,  d} \langle \Omega\rangle \psi_{2}\big)      \Big\|_{L^{1}_{t}L^{2}_{x}}
\]

\[
\lesssim  \sum_{d < c_{2} \mu} \mu^{-\frac{1}{2}} \Big\|  \Big( \sum_{\omega_{1}, \omega_{2},\omega_{3}} \big\| S^{\omega_1,\pm}_{\mu, \bullet \leq \min\{d,\mu\}} Q_{a,b} \big( B^{\omega_2}_{(\frac{d}{\mu})^{\frac{1}{2}}} S^{a}_{1, \bullet \leq d} \psi_{1}, B^{\omega_3}_{(\frac{d}{\mu})^{\frac{1}{2}}} S^{b}_{1, d} \langle \Omega\rangle  \psi_{2}\big)\big\|_{L^{2}_{x}}^{2} \Big)^{\frac{1}{2}}        \Big\|_{L^{1}_{t}} 
\]

\[
\lesssim \sum_{d < c_{2} \mu} \mu^{\frac{1}{2}-\alpha+\frac{3}{4}-}\Big(\frac{d}{\mu}\Big)^{\frac{\alpha}{2}+\frac{1}{4}-} \Big\| \sup_{\omega_{2}} \| B^{\omega_2}_{(\frac{d}{\mu})^{\frac{1}{2}}} S^{a}_{1, \bullet \leq d} \psi_{1}\|_{L^{4+}_{x}} (\sum_{\omega_{3}} \| B^{\omega_3}_{(\frac{d}{\mu})^{\frac{1}{2}}} S^{b}_{1, d}  \langle \Omega\rangle  \psi_{2}\|_{L^{2}_{x}}^{2})^{\frac{1}{2}}\Big\|_{L^{1}_{t}}
\]

\[
\lesssim \sum_{d < c_{2} \mu} \mu^{\frac{1}{2}-\alpha+\frac{3}{4}-}\Big(\frac{d}{\mu}\Big)^{\frac{\alpha}{2}+\frac{1}{2}-} \| S_{1}\psi_{1}\|_{\widetilde{F}_{\Omega}} \| S_{1,d}^{b}\langle \Omega\rangle \psi_{2}\|_{L^{2}_{t}L^{2}_{x}} \]
\begin{equation}\label{equation21020}
\lesssim \sum_{d < c_{2} \mu} \mu^{\frac{1}{2}-\alpha+\frac{3}{4}-}\Big(\frac{d}{\mu}\Big)^{\frac{\alpha}{2}+\frac{1}{2}-} d^{-\frac{1}{2}}  \| S_{1}\psi_{1}\|_{\widetilde{F}_{\Omega}} \| S_{1} \psi_{2}\|_{\widetilde{F}_{\Omega}} \lesssim \mu^{\frac{1}{2}}\| S_{1}\psi_{1}\|_{\widetilde{F}_{\Omega}} \| S_{1} \psi_{2}\|_{\widetilde{F}_{\Omega}},
\end{equation}
as before, here we choose $\alpha\in (0,1/5)$ when $a \cdot b > 0$ and $\alpha = 1$, $\mu\sim 1$ when $a\cdot b < 0$. 
\[
\mathcal{I}_{2,2} \lesssim \sum_{d < c_{2} \mu} \mu^{-\frac{1}{2}}  \Big\|    \sum_{\scriptsize \begin{array}{c}
|\omega_{1} \mp a\cdot\omega_{2}| \sim (\frac{d}{\mu})^{\frac{1}{2}}\\
|\omega_{2} - ab\cdot \omega_{3}| \sim (\frac{d}{\mu})^{\frac{1}{2}}\\
\end{array}}  S^{\omega_1,\pm}_{\mu, \bullet \leq \min\{d,\mu\}} Q_{a,b} \big( B^{\omega_2}_{(\frac{d}{\mu})^{\frac{1}{2}}} S^{a}_{1,\bullet \leq d} \langle \Omega\rangle \psi_{1}, B^{\omega_3}_{(\frac{d}{\mu})^{\frac{1}{2}}} S^{b}_{1,  d} \psi_{2}\big)      \Big\|_{L^{1}_{t}L^{2}_{x}}
\]

\[
\lesssim \sum_{d \leq c_{2}\mu}
\mu^{\frac{1}{2}-\alpha+ \frac{3}{p}} \Big(\frac{d}{\mu}\Big)^{\frac{\alpha}{2} + \frac{1}{p}} \mu^{-\frac{1}{2}}   d^{-\frac{1.1}{p}} \| S_{1}\langle \Omega\rangle \psi_{1}\|_{L^{\infty}_{t}L^{2}_{x}} \| S_{1} \psi_{2}\|_{Z_{\Omega, 1}}
\]
\begin{equation}\label{equation21026}
\lesssim \mu^{\frac{1.9}{p}-\alpha} \| S_{1}\psi_{1}\|_{\widetilde{F}_{\Omega}} \| S_{1}\psi_{2}\|_{\widetilde{F}_{\Omega}}.
\end{equation}
In the proof of (\ref{equation21026}), $\alpha$ is chosen inside $(0, 0.4/p)$,
hence from (\ref{equation21020}) and (\ref{equation21026}), we have
\begin{equation}\label{equation21021}
\| V\mathcal{I}_{2}\|_{F^{1/2}_{\Omega}} \lesssim  \mu^{\frac{1}{p}} \| S_{1}\psi_{1}\|_{\widetilde{F}_{\Omega}} \| S_{1}\psi_{2}\|_{\widetilde{F}_{\Omega}}.
\end{equation}
By symmetry, the term $\mathcal{I}_{3}$ can be estimated in the same way as $\mathcal{I}_{2}$, hence from (\ref{equation51040}) and (\ref{equation21021}), we could see that (\ref{equation21010}) holds.

The proof of (\ref{equation21023}) follows from the same strategy used in the proof of (\ref{equation21023}), however, as the null structure of $\tilde{Q}_{a,b}$ only guarantee  us to gain $(d/\mu)^{1/2}$ but not $\mu$ any more, hence in above estimates $\alpha$ has to be $1$. As a consequence, one might see that the power of $\mu$ in (\ref{equation21025}) and (\ref{equation21026}) is negative, hence will cause problem in the summation with respect to $\mu$. However, as we are in the  High$\times$High interaction case, the angle between the spatial frequencies of two inputs should less than $\mu$ when both of inputs are very close to the cone. The angular sector decomposition we did in the wide angle decomposition will be very large when modulation $d$ close to $\mu$, i.e when $(d/\mu)^{1/2} \gg \mu$. Hence it suggests us to do the following:
\begin{enumerate}
\item[(i)] when $(d/\mu)^{1/2} \lesssim \mu$, we still use the wide angle decomposition.\\

\item[(ii)] When $\mu \lesssim (d/\mu)^{1/2} \lesssim 1$, we use a smaller sector decomposition, which is to use sector with angle $\mu$ instead of $(d/\mu)^{1/2}$. Notice that here we  only have $\log\mu$ cases to consider, hence we don't worry about the summation problem with respect to $d$ but with the price of loss $\log \mu$.
\end{enumerate}
Another thing to notice is that we won't loss $\mu^{1/2}$ from the regularity in (\ref{equation21023}), i.e we gain $\mu^{1/2}$ in the upper bounds of types (\ref{equation21025}) and (\ref{equation21026}). In the case (i), the upper bound of type (\ref{equation21025}) would be  $\mu^{ 1-}$ and the upper bound of type (\ref{equation21026}) would be $ \mu^{\frac{1.7}{p}+\frac{1}{2}} $; in the case (ii), the upper bound of type (\ref{equation21025}) would be $\mu^{  1-} \log\mu $, in the proof of type (\ref{equation21026}), the major difference would be the number of sectors putted in $Z_{\Omega, \lambda}$ space and it will change from $\mu^{-1/2}$ to $\mu d^{-1/2}$, hence the new upper bound would be $\mu^{\frac{1.7}{p} + \frac{1}{2}}\log\mu$.
Therefore, we can see that for all cases, the exponents of $\mu$ remains positive, hence (\ref{equation21023}) holds.
\end{proof}

\subsection{Low$\times$ High and High $\times$Low:  near the cone}

\begin{lemma}[Low$\times$ High and High $\times$ Low: near the cone]\label{lowhighnear}
For $\psi_{1}, \psi_{2}\in \widetilde{F}_{\Omega}$ and $\phi\in F^{1/2}_{\Omega}$, we have the 
following estimates:

\begin{equation}\label{equation55007}
\| V   S_{\lambda, \bullet < c_{1}\mu}^{-b} Q_{a,b} (S_{\mu} \psi_{1}, S_{\lambda, \bullet < c_{2} \mu}^{b} \psi_{2})\|_{F^{1/2}_{\Omega}}
 \lesssim
 \| S_{\mu}\psi_{1}\|_{\widetilde{F}_{\Omega}} \| S_{\lambda} \psi_{2}\|_{\widetilde{F}_{\Omega}},
 \end{equation}

\begin{equation}\label{equation55009}
\| V_{a}   S_{\lambda, \bullet < c_{1}\mu}^{a} \tilde{Q}_{a,b} (S_{\lambda, \bullet < c_{2} \mu}^{a} \phi, S_{\mu} \psi_{1}) \big]\|_{\widetilde{F}_{\Omega}}
 \lesssim
\| S_{\lambda}\phi\|_{F^{1/2}_{\Omega}} \| S_{\mu} \psi_{1}\|_{\widetilde{F}_{\Omega}},
\end{equation}

\begin{equation}\label{equation21000}
\| V_{a}   S_{\lambda, \bullet < c_{1}\mu}^{a} \tilde{Q}_{a,b} (S_{\mu} \phi, S_{\lambda, \bullet < c_{2} \mu}^{b} \psi_{1}) \big]\|_{\widetilde{F}_{\Omega}}
 \lesssim
 \| S_{\mu}\phi\|_{F^{1/2}_{\Omega}} \| S_{\lambda} \psi_{1}\|_{\widetilde{F}_{\Omega}}.
  \end{equation}
\end{lemma}

\begin{proof}
As the function spaces are scale invariant, it would be sufficient to prove above estimates under assumption $\lambda =1$. Let's do the following decomposition:

\[
S_{1, \bullet < c_{1}\mu}^{-b} Q_{a,b} (S_{\mu} \psi_{1}, S_{1, \bullet < c_{2} \mu}^{b} \psi_{2})
 = \mathcal{J}_{1} + \mathcal{J}_{2} + \mathcal{J}_{3},
\]
where

\[
\mathcal{J}_{1} =  \sum_{ d\leq \mu} S^{-b}_{1,\bullet \leq d} {Q}_{a,b} (S_{\mu, d}^{a} \psi_{1}, S_{1, \bullet \leq d}^{b} \psi_{2}), \mathcal{J}_{2} = \sum_{d < c_{1}\mu} S^{-b}_{1, d} {Q}_{a,b} (S_{\mu,\bullet \leq  \min\{d, \mu\}}^{a} \psi_{1}, S_{1, \bullet \leq \min\{d, c_{2}\mu\}}^{b} \psi_{2})
, \]
\[
 \mathcal{J}_{3} =\sum_{d < c_{2} \mu}
S^{-b}_{1, \bullet \leq d} {Q}_{a,b} (S_{\mu,\bullet \leq  \min\{d,\mu\}}^{a} \psi_{1}, S_{1, d}^{b} \psi_{2}).
\]

\[
\| V \mathcal{J}_{1}\|_{F^{1/2}_{\Omega}} \lesssim  \mathcal{J}_{1,1} + \mathcal{J}_{1,2}:= \sum_{ d\leq \mu} \|S^{-b}_{1,\bullet \leq d} {Q}_{a,b} (S_{\mu, d}^{a} \langle \Omega \rangle \psi_{1}, S_{1, \bullet \leq d}^{b} \psi_{2})\|_{L^{1}_{t}L^{2}_{x}}
\]

\[
+ \sum_{ d\leq \mu} \|S^{-b}_{1,\bullet \leq d} {Q}_{a,b} (S_{\mu, d}^{a} \psi_{1}, S_{1, \bullet \leq d}^{b}  \langle \Omega \rangle  \psi_{2})\|_{L^{1}_{t}L^{2}_{x}}.
\]

\[
\mathcal{J}_{1,1} \lesssim \sum_{d \leq \mu} \Big\| \sum_{\tiny \begin{array}{c}
|\omega_{1} + \omega_{3}| \sim (\frac{d}{\mu})^{\frac{1}{2}}\\
|\omega_{2} - ab \cdot \omega_{3}| \sim (\frac{d}{\mu})^{\frac{1}{2}}\\
\end{array}
}\normalsize B^{\omega_{1}}_{(\frac{d}{\mu})^{\frac{1}{2}}} S^{-b}_{1, \bullet < d} Q_{a, b} (S^{\omega_{2}, a}_{\mu, d}\langle \Omega\rangle \psi_{1}, B^{\omega_{3}}_{(\frac{d}{\mu})^{1/2}} S^{b}_{1, \bullet < d} \psi_{2}) \Big\|_{L^{1}_{t}L^{2}_{x}}
\]

\[
\lesssim \sum_{d \leq \mu} \Big\| \big( \sum_{\omega_{1}, \omega_{2}, \omega_{3}} \|  B^{\omega_{1}}_{(\frac{d}{\mu})^{\frac{1}{2}}} S^{-b}_{1, \bullet < d} Q_{a, b} (S^{\omega_{2}, a}_{\mu, d} \langle \Omega\rangle \psi_{1}, B^{\omega_{3}}_{(\frac{d}{\mu})^{1/2}} S^{b}_{1, \bullet < d} \psi_{2}) \|_{L^{2}}^{2}\big)^{\frac{1}{2}} \Big\|_{L^{1}_{t}}
\]

\[
\lesssim \sum_{d \leq \mu} \Big(\frac{d}{\mu}\Big)^{\frac{1}{2}} 
\Big\| \big(\sum_{ \omega_{2}, \omega_{3}} \| B^{\omega_{3}}_{(\frac{d}{\mu})^{1/2}} S^{b}_{1, \bullet < d} \psi_{2}\|_{L^{4+}_{x}}^{2}  \|S^{\omega_{2}, a}_{\mu, d}\langle \Omega\rangle \psi_{1}\|_{L^{4-}}^{2} \big)^{1/2} \Big\|_{L^{1}_{t}}
\]

\begin{equation}\label{equation210410}
\lesssim \sum_{d \leq \mu} \Big(\frac{d}{\mu}\Big)^{1-} \mu^{\frac{3}{4}-} d^{-\frac{1}{2}} \| S_{1} \psi_{2}\|_{\widetilde{F}_{\Omega}} \| S_{\mu} \psi_{1}\|_{\widetilde{F}_{\Omega}}\lesssim \mu^{\frac{1}{4}-} \| S_{1} \psi_{2}\|_{\widetilde{F}_{\Omega}} \| S_{\mu} \psi_{1}\|_{\widetilde{F}_{\Omega}}.
\end{equation}

\[
\mathcal{J}_{1,2} \lesssim \sum_{d\lesssim \mu} \Big(\frac{d}{\mu}\Big)^{\frac{1}{2}} \sup_{\omega} \| S^{\omega,a}_{\mu, d} \psi_{1}\|_{L^{1}_{t}L^{\infty}_{x}} \| S^{b}_{1, \bullet < d} \langle \Omega\rangle \psi_{2}\|_{L^{\infty}_{t}L^{2}_{x}}
\]

\begin{equation}\label{equation21041}
\lesssim  \sum_{d\lesssim \mu} \Big(\frac{d}{\mu}\Big)^{\frac{1}{2}-\frac{0.1}{p}} \mu^{\frac{1}{2}}\| S_{\mu}\psi_{1}\|_{\widetilde{F}_{\Omega}} \| S_{1} \psi_{2}\|_{\widetilde{F}_{\Omega}}\lesssim \mu^{1/2} \| S_{\mu}\psi_{1}\|_{\widetilde{F}_{\Omega}} \| S_{1}\psi_{2}\|_{\widetilde{F}_{\Omega}}.
\end{equation}

\[
\| V\mathcal{J}_{2}\|_{\langle \Omega\rangle^{-1}X^{1/2,1}_{1}} \lesssim \mathcal{J}_{2,1}+ \mathcal{J}_{2,2}:= \sum_{ d < c_{1}\mu} d^{-\frac{1}{2}} \|  S^{-b}_{1, d} {Q}_{a,b} (S_{\mu,\bullet \leq  \min\{d, \mu\}}^{a} \psi_{1}, S_{1, \bullet \leq \min\{d, c_{2}\mu\}}^{b} \langle\Omega\rangle\psi_{2})\|_{L^{2}_{t}L^{2}_{x}}
\]
\[
+ \sum_{ d < c_{1}\mu} d^{-\frac{1}{2}} \|  S^{-b}_{1, d} {Q}_{a,b} (S_{\mu,\bullet \leq  \min\{d,\mu\}}^{a}\langle\Omega\rangle \psi_{1}, S_{1, \bullet \leq \min\{d, c_{2}\mu\}}^{b} \psi_{2})\|_{L^{2}_{t}L^{2}_{x}}
\]

\[
\mathcal{J}_{2,1} \lesssim \sum_{ d < c_{1}\mu} d^{-\frac{1}{2}} \big\| \sum_{\tiny \begin{array}{c}
|\omega_{1} + \omega_{3}| \sim (\frac{d}{\mu})^{\frac{1}{2}}\\
|\omega_{2} - ab \omega_{3}| \sim (\frac{d}{\mu})^{\frac{1}{2}}\\
\end{array}
}\normalsize B^{\omega_{1}}_{(\frac{d}{\mu})^{\frac{1}{2}}} S^{-b}_{1, d} Q_{a, b} \big(S^{\omega_{2}, a}_{\mu, \bullet \leq \min\{d,\mu\}}\psi_{1}, B^{\omega_{3}}_{(\frac{d}{\mu})^{1/2}} S^{b}_{1, \bullet < \min\{ c_{2} \mu, d\}}\langle \Omega\rangle  \psi_{2}\big) \big\|_{L^{2,2}_{t,x}}\]
\[
\lesssim  \sum_{ d < c_{1}\mu} d^{-\frac{1}{2}}  \Big(\frac{d}{\mu}\Big)^{\frac{1}{2}} \sup_{\omega} \| S^{\omega,a}_{\mu, \bullet \leq \min\{d,\mu\}} \psi_{1}\|_{L^{2}_{t}L^{\infty}_{x}} \| S^{b}_{1, \bullet < \min\{ c_{2} \mu, d\}}\langle \Omega\rangle  \psi_{2}\|_{L^{\infty}_{t}L^{2}_{x}}
\]
\begin{equation}\label{equation55000}
\lesssim  \sum_{ d < c_{1}\mu} d^{-\frac{1}{2}}  \Big(\frac{d}{\mu}\Big)^{1-} \mu \| S_{\mu} \psi_{1}\|_{\widetilde{F} _{\Omega}} \| S_{1} \psi_{2}\|_{\widetilde{F}_{\Omega}} \lesssim \mu^{\frac{1}{2}} \| S_{\mu}\psi_{1}\|_{\widetilde{F} _{\Omega}} \| S_{1} \psi_{2}\|_{\widetilde{F}_{\Omega}},
\end{equation}

\[
\mathcal{J}_{2,2}\lesssim   \sum_{ d < c_{1}\mu} d^{-\frac{1}{2}}  \Big(\frac{d}{\mu}\Big)^{\frac{1}{2}} \big\| \big( \sum_{\omega_{2}, \omega_{3}} \| B^{\omega_{2}}_{(\frac{d}{\mu})^{1/2}} S^{a}_{\mu, \bullet \leq \min\{d,\mu\}}\langle \Omega\rangle \psi_{1}\|_{L^{4-}_{x}}^{2} \| B^{\omega_{3}}_{(\frac{d}{\mu})^{1/2}} S^{b}_{1, \bullet < \min\{ c_{2} \mu, d\}} \psi_{2})\|_{L^{4+}_{x}}^{2} \big)^{1/2} \big\|_{L^{2}_{t}}
\]

\begin{equation}\label{equation55001}
\lesssim  \sum_{ d < c_{1}\mu} d^{-\frac{1}{2}}  \Big(\frac{d}{\mu}\Big)^{1-} \mu^{\frac{3}{4}-}  \| S^{a}_{\mu, \bullet \leq \min\{d,\mu\}} \langle \Omega\rangle \psi_{1}\|_{L^{\infty}_{t}L^{2}_{x}} \| S_{1}\psi_{2} \| _{\widetilde{F}_{\Omega}} \lesssim \mu^{\frac{1}{4}-} \| S_{\mu}\psi_{1}\|_{\widetilde{F}_{\Omega}} \| S_{1}\psi_{2}\|_{\widetilde{F}_{\Omega}}.
\end{equation}

\vspace{1\baselineskip}
\[
\| V \mathcal{J}_{2}\|_{Z_{\Omega, 1}} \lesssim 
\sum_{d < c_{1}\mu}  d^{-1+\frac{1.1}{p}} \int \sup_{\omega} \| S_{1,d}^{\omega} {Q}_{a,b} (B^{\omega_{2}}_{(\frac{d}{\mu})^{1/2}} S_{\mu,\bullet \leq  \min\{d,\mu\}}^{a} \psi_{1}, B^{\omega_{3}}_{(\frac{d}{\mu})^{1/2}}  S_{1, \bullet \leq \min\{d, c_{2}\mu\}}^{b} \psi_{2})\|_{L^{p}_{x}} \, d \, t
\]
\begin{equation}\label{equation55002}
\lesssim \sum_{d < c_{1}\mu} d^{-1+\frac{1.1}{p}} \mu \Big( \frac{d}{\mu}\Big)^{-\frac{1}{p}+ \frac{3}{2}-} \|  S_{\mu}\psi_{1}\|_{\widetilde{F}_{\Omega}} \|  S_{1} \psi_{2}\|_{\widetilde{F}_{\Omega}}
\lesssim \mu^{\frac{1}{p}} \| S_{\mu}\psi_{1}\|_{\widetilde{F}_{\Omega}} \| S_{1}\psi_{2}\|_{\widetilde{F}_{\Omega}}.
\end{equation}

\vspace{1\baselineskip}

\[
\| V\mathcal{J}_{3}\|_{F^{1/2}} \lesssim \mathcal{J}_{3,1}+ \mathcal{J}_{3,2}:= \sum_{ d < c_{2}\mu} \|  S^{-b}_{1, \bullet \leq d} {Q}_{a,b} (S_{\mu,\bullet \leq   \min\{d, \mu\}}^{a} \psi_{1}, S_{1, d}^{b} \langle \Omega \rangle \psi_{2})\|_{L^{1}_{t}L^{2}_{x}}\]

\[
+  \sum_{ d < c_{2}\mu} \|  S^{-b}_{1, \bullet \leq d} {Q}_{a,b} (S_{\mu,\bullet \leq   \min\{d, \mu\}}^{a} \langle \Omega \rangle \psi_{1}, S_{1, d}^{b}  \psi_{2})\|_{L^{1}_{t}L^{2}_{x}},
\]

\[
\mathcal{J}_{3,1}\lesssim \sum_{ d < c_{2}\mu} \big\| \sum_{\tiny \begin{array}{c}
|\omega_{1} + \omega_{3}| \sim (\frac{d}{\mu})^{\frac{1}{2}}\\
|\omega_{2} - ab \cdot \omega_{3}| \sim (\frac{d}{\mu})^{\frac{1}{2}}\\
\end{array}
}\normalsize B^{\omega_{1}}_{(\frac{d}{\mu})^{\frac{1}{2}}} S^{-b}_{1, \bullet \leq d} Q_{a, b} ( S^{\omega_{2},a}_{\mu, \bullet \leq  \min\{d, \mu\}} \psi_{1}, B^{\omega_{3}}_{(\frac{d}{\mu})^{1/2}} S^{b}_{1, d}\langle \Omega \rangle \psi_{2}) \big\|_{L^{1}_{t}L^{2}_{x}}
\]

\[
\lesssim  \sum_{ d < c_{2}\mu} \mu^{\frac{3}{4}-} \Big(\frac{d}{\mu}\Big)^{\frac{3}{4}-} \sup_{\omega_{2}} \| B^{\omega_{2}}_{(\frac{d}{\mu})^{1/2}} S_{\mu, \leq  \min\{d, \mu\}}^{a} \psi_{1}\|_{L^{2}_{t}L^{4+}_{x}} \| S_{1, d}^{b}\langle \Omega \rangle \psi_{2} \|_{L^{2}_{t}L^{2}_{x}}
\]

\begin{equation}\label{equation55003}
 \lesssim \sum_{d < c_{2}\mu}  \Big(\frac{d}{\mu}\Big)^{1-} d^{-\frac{1}{2}} \mu \|S_{\mu} \psi_{1}\|_{\widetilde{F}_{\Omega}} \| S_{1} \psi_{2} \|_{\widetilde{F}_{\Omega}} \lesssim \mu^{\frac{1}{2}}  \|S_{\mu} \psi_{1}\|_{\widetilde{F}_{\Omega}} \| S_{1} \psi_{2} \|_{\widetilde{F}_{\Omega}}.
\end{equation}

\[
\mathcal{J}_{3,2} \lesssim \sum_{d \lesssim \mu^{3}} 
\mu^{\frac{3}{p}} \Big(\frac{d}{\mu}\Big)^{\frac{1}{2}+\frac{1}{p}} \mu^{-\frac{1}{2}} d^{-\frac{1.1}{p}} \| S_{\mu}\langle \Omega \rangle\psi_{1} \|_{L^{\infty}_{t}L^{2}_{x}} \| S_{1}\psi_{2}\|_{Z_{\Omega, 1}} +  \sum_{\mu^{3} \leq d < c_{2}\mu} \mathcal{J}_{3,2}^{d}
\]

\begin{equation}\label{equation55004}
 \lesssim \mu^{\frac{1}{2}+\frac{1.7}{p}} \| S_{\mu}\psi_{1} \|_{\widetilde{F}} \| S_{1} \psi_{2}\|_{\widetilde{F}_{\Omega}} +   \sum_{\mu^{3} \leq d < c_{2}\mu} \mathcal{J}_{3,2}^{d},
\end{equation}
where
\[
\mathcal{J}_{3,2}^{d} =  \big\| \sum_{\tiny \begin{array}{c}
|\omega_{1} + \omega_{3}| \sim \mu\\
|\omega_{2} - ab \cdot \omega_{3}| \sim (\frac{d}{\mu})^{\frac{1}{2}}\\
\end{array}
}\normalsize B^{\omega_{1}}_{\mu} S^{-b}_{1, \bullet \leq d} Q_{a, b} (B^{\omega_{2}}_{(\frac{d}{\mu})^{1/2}} S^{a}_{\mu, \bullet \leq  \min\{d, \mu\}}\langle \Omega\rangle \psi_{1}, B^{\omega_{3}}_{\mu} S^{b}_{1, d} \psi_{2}) \big\|_{L^{1}_{t}L^{2}_{x}},
\]
here we utilize the information that the in the Low $\times$ High interaction case, the angle between the high input $\psi_{2}$ and the output shouldn't be too big and less than $\mu$. Hence for the output and $\psi_{2}$, we can decompose and localize them into a smaller sector with size $\mu$ and localize $\psi_{1}$ into  a bigger sector with size $(d/\mu)^{1/2}$, notice that the summation with respect to $\omega_{1}, \omega_{2}, \omega_{3}$ now is not in diagonal, for each fixed sector $\omega_{2}$, correspondingly, there are  $d^{1/2}\mu^{-3/2}$  sectors $\omega_{1}$ and $\omega_{3}$.
For fixed time $t$ and modulation $d$, we have 
\[
\| \sum_{\tiny \begin{array}{c}
|\omega_{1} + \omega_{3}| \sim \mu\\
|\omega_{2} - ab \cdot \omega_{3}| \sim (\frac{d}{\mu})^{\frac{1}{2}}\\
\end{array}
}\normalsize B^{\omega_{1}}_{\mu} S^{-b}_{1, \bullet \leq d} Q_{a, b} (B^{\omega_{2}}_{(\frac{d}{\mu})^{1/2}} S^{a}_{\mu, \bullet \leq \min\{d,\mu\}} \langle \Omega \rangle\psi_{1}, B^{\omega_{3}}_{\mu} S^{b}_{1, d} \psi_{2})
\|_{L^{2}_{x}} \lesssim  \sup_{h: \| h(x)\|_{L^{2}_{x}}=1}   \]
\[
\Big( \frac{d}{\mu}\Big)^{\frac{1}{2}}  \sum_{\omega_{2}} \|  S^{\omega_{2}, a}_{\mu, \bullet \leq \min\{d,\mu\}} \langle \Omega\rangle \psi_{1}\|_{L^{2}_{x}}  \| \sum_{\tiny \begin{array}{c}
|\omega_{1} + \omega_{3}| \sim \mu\\
|\omega_{2} - ab \cdot \omega_{3}| \sim (\frac{d}{\mu})^{\frac{1}{2}}\\
\end{array}}
\normalsize 
B^{\omega_{2}}_{(\frac{d}{\mu})^{\frac{1}{2}}} P_{\bullet \lesssim \mu} T(B^{\omega_{1}}_{\mu}P_{\bullet \lesssim 1} h,  B^{\omega_{3}}_{\mu}S^{b}_{1,d}\psi_{2})\|_{L^{2}_{x}}
\]

\[
\lesssim\Big( \frac{d}{\mu}\Big)^{\frac{1}{2}+\frac{1}{p}} \mu^{\frac{3}{p}}  \|  S^{a}_{\mu, \bullet \leq d} \langle \Omega\rangle \psi_{1}\|_{L^{2}_{x}} \sup_{\omega_{3}} \| B^{\omega_{3}}_{\mu}  S^{b}_{1,d}\psi_{2}\|_{L^{p}_{x}},
\]
here $T$ is some bilinear operator with $L^{1}$ kernel, hence
\[
\sum_{\mu^{3} \leq d < c_{2}\mu} \mathcal{J}_{3,2}^{d} \lesssim  \sum_{\mu^{3} \leq d < c_{2}\mu}   \Big( \frac{d}{\mu}\Big)^{\frac{1}{2}+\frac{1}{p}}  \mu^{1+\frac{3}{p}} d^{-\frac{1}{2}-\frac{1.1}{p}} \|  S^{a}_{\mu, \bullet \leq d} \langle \Omega \rangle \psi_{1}\|_{L^{\infty}_{t}L^{2}_{x}} \| S_{1} \psi_{2}\|_{Z_{\Omega, 1}} 
\]
\begin{equation}\label{equation54000}
\lesssim \mu^{\frac{1}{2}+\frac{1.7}{p}-} \| S_{\mu}\psi_{1}\|_{\widetilde{F}_{\Omega}} \| S_{1}\psi_{2} \|_{\widetilde{F}_{\Omega}}.
\end{equation}
Combine the results: (\ref{equation210410}), (\ref{equation21041}), (\ref{equation55000}), (\ref{equation55001}), (\ref{equation55002}), (\ref{equation55003}), (\ref{equation55004}) and (\ref{equation54000}), we could see the desired estimate (\ref{equation55007}) holds. The proof of desired estimates (\ref{equation55009}) and (\ref{equation21000}) are very similar. Notice that for the (\ref{equation21000}), as we will loss $\mu^{1/2}$ due to the regularity of $\phi$, we have to gain more, one important observation is that as now $\phi$ is at low frequency, hence the angle between the output and the Dirac part should be less than $\mu$ and it's exactly the symbol of null structure of $\tilde{Q}_{a,b}$ when $a\cdot b > 0$, hence we can play the trick of weight between $\mu$ and $(d/\mu)^{1/2}$ like we did before, and when $a\cdot b < 0$, we know that the type of term as the left hand side of  (\ref{equation55007}) is used when $\mu\sim 1$ hence loss $\mu$ is not a issue.
 Here we omit the detail and only give a upper bound associated with (\ref{equation21000}), which is $\mu^{{1}/{p}}\| S_{\mu}\phi\|_{F^{1/2}_{\Omega}} \| S_{\lambda} \psi_{2}\|_{\widetilde{F}_{\Omega}}.
$
\end{proof}

\subsection{Commutator terms arise from High $\times$ High interaction}

\begin{lemma}\label{commutator}
For $\psi_{1}, \psi_{2}\in \widetilde{F}_{\Omega}$ and $\phi\in F^{1/2}_{\Omega}$, we have the 
following estimates:
\begin{equation}\label{equation22002}
\sum_{\sigma : \mu < \sigma \leq 1} \| W(P_{\mu} S_{\sigma, \sigma} V\,Q_{a,b} (S_{1} \psi_{1}, S_{1} \psi_{2}) ) \|_{F^{1/2}_{\Omega}}  \lesssim \mu^{\frac{3}{p}-\frac{1}{2}} \| S_{1} \psi_{1} \|_{\widetilde{F}_{\Omega}} \| S_{1} \psi_{2} \|_{\widetilde{F}_{\Omega}},
  \end{equation}

\begin{equation}\label{equation22001}
\sum_{\sigma : \mu < \sigma \leq 1} W(P_{\mu} S_{\sigma, \sigma} V_{a} \tilde{Q}_{a,b} (S_{1} \phi, S_{1} \psi_{1}) )\|_{\widetilde{F}_{\Omega}} \lesssim \mu^{\frac{1}{2}} \| S_{1}\phi\|_{F^{1/2}_{\Omega}} \| S_{1}\psi_{1}\|_{\widetilde{F}_{\Omega}}.
\end{equation}
\end{lemma}
\begin{proof}
The  proof of (\ref{equation22001}) is very straightforward, as we can gain $\mu^{3/2}$ from the Sobolev embedding, while only loss $\mu$ comes from the modulation hence it's easy to see (\ref{equation22001}) holds. More precisely 
\[
\sum_{\sigma : \mu < \sigma \leq 1} \| W(P_{\mu} S_{\sigma, \sigma} V_{a} \tilde{Q}_{a,b} (S_{1} \phi, S_{1} \psi_{1}) )\|_{\widetilde{F}_{\Omega}}  \lesssim \sum_{\sigma : \mu < \sigma \leq 1} \| \langle \Omega \rangle(P_{\mu} S_{\sigma, \sigma} V_{a} \tilde{Q}_{a,b} (S_{1} \phi, S_{1} \psi_{1}) )\|_{L^{\infty}_{t}L^{2}_{x}}
\]
\[
\lesssim \sum_{\sigma : \mu < \sigma \leq 1} \mu^{3/2} \sigma^{-1} \| S_{1}\langle \Omega \rangle \phi\|_{L^{\infty}_{t} L^{2}_{x}} \| S_{1}\langle \Omega \rangle\psi_{1}\|_{L^{\infty}_{t}L^{2}_{x}} \lesssim \mu^{\frac{1}{2}} \| S_{1}\phi\|_{F^{1/2}_{\Omega}} \| S_{1}\psi_{1}\|_{\widetilde{F}_{\Omega}}.
\]

While to prove (\ref{equation22002}), if we stick to above rough estimate, we will loss $\mu^{-3/2}$ from summation in modulation while gain $\mu^{3/2}$ from the Sobolev embedding, hence we will have problem in the summation part with respect to $\mu$. However, one thing to notice is that as the output spatial frequency is less than $\mu$, when the space-time frequencies of two inputs are both close to the cone, then the angle between the spatial frequencies of two inputs should less than $\mu$. More precisely,  for fixed modulation $\sigma$,  when $a\cdot  b > 0$, we have 
\begin{equation}\label{equation22062}
P_{\mu} S_{\sigma, \sigma}Q_{a,b} (S_{1} \psi_{1}, S_{1} \psi_{2}) = \mathcal{H}_{1} + \mathcal{H}_{2} + \mathcal{H}_{3} + \mathcal{H}_{4}
\end{equation}
\[
\quad\mathcal{H}_{1} = P_{\mu} S_{\sigma, \sigma} Q_{a,b} (S_{1,\bullet \geq 4 c}^{a} \psi_{1}, S_{1, \bullet < 2 c}^{b} \psi_{2}),\quad \mathcal{H}_{2} = P_{\mu} S_{\sigma, \sigma} Q_{a,b} (S_{1,\bullet < 4 c}^{a} \psi_{1}, S_{1, \bullet < 2 c}^{b} \psi_{2})\]

\[
\mathcal{H}_{3} = P_{\mu} S_{\sigma, \sigma} Q_{a,b} (S_{1, \bullet \geq c }^{a} \psi_{1}, S_{1, \bullet \geq 2 c}^{b} \psi_{2}), \quad \mathcal{H}_{4} = P_{\mu} S_{\sigma, \sigma} Q_{a,b} (S_{1, \bullet \leq c }^{a} \psi_{1}, S_{1, \bullet \geq 2 c}^{b} \psi_{2})
\]

For term $\mathcal{H}_{2}$, we could gain $\mu$ from the null structure as $a\cdot b > 0$ and the angle between two inputs is less than $\mu$, hence

\[
\| W( V \mathcal{H}_{2})\|_{F^{1/2}_{\Omega}} \lesssim \sum_{\sigma : \mu < \sigma \leq 1} \mu^{\frac{5}{2}} \sigma^{-2+ \frac{1}{2}} 
\| (S_{1,\bullet < 4 c} \langle \Omega \rangle \psi_{1})\|_{L^{\infty}_{t}L^{2}_{x}} \| (S_{1,\bullet < 2 c} \langle \Omega \rangle  \psi_{2})\|_{L^{\infty}_{t}L^{2}_{x}}
\]
\begin{equation}\label{equation22063}
\lesssim \mu \| S_{1}\psi_{1}\|_{\widetilde{F}_{\Omega}} \| S_{1}\psi_{2}\|_{\widetilde{F}_{\Omega}},
\end{equation}

\[
\| W( V \mathcal{H}_{3})\|_{F^{1/2}_{\Omega}} \lesssim \sum_{\sigma : \mu < \sigma \leq 1} \sigma^{-\frac{1}{2}} \| \langle \Omega \rangle \mathcal{H}_{3} \|_{L^{1}_{t}L^{2}_{x}}  \lesssim \sum_{\sigma_{2}\geq 2 c}\mu^{-\frac{1}{2}+\frac{3}{4}-}
\| S_{1, \bullet 
\geq c}\psi_{1}\|_{L^{2}_{t}L^{4+}_{x}} \| S_{1, \sigma_{2}}\langle \Omega \rangle \psi_{2}\|_{L^{2}_{t}L^{2}_{x}}
\]
\begin{equation}\label{equation22060}
+ \sum_{\sigma_{1}\geq  c}\mu^{-\frac{1}{2}+\frac{3}{4}-}
\| S_{1, \sigma_{1}} \langle \Omega\rangle \psi_{1}\|_{L^{2}_{t}L^{2}_{x}} \| S_{1, \bullet \geq 2 c} \psi_{2}\|_{L^{2}_{t}L^{4+}_{x}}  \lesssim \mu^{\frac{1}{4}-}
\| S_{1} \psi_{1}\|_{\widetilde{F}_{\Omega}} \| S_{1}\psi_{2}\|_{\widetilde{F}_{\Omega}}.
\end{equation}

For terms $\mathcal{H}_{1}$ and $\mathcal{H}_{4}$, because  the distance of the two modulations has constant lower bound, and as $a \cdot b >0$, we can derive that $\sigma \sim 1$, hence
\[
\| W( V (\mathcal{H}_{1} + \mathcal{H}_{4}))\|_{F^{1/2}_{\Omega}} \lesssim \sum_{\sigma \sim 1} \sigma^{-\frac{3}{2}} \| \langle \Omega\rangle (\mathcal{H}_{1} + \mathcal{H}_{4}) \|_{L^{\infty}_{t}L^{2}_{x}}  
\]
\begin{equation}\label{equation22061}
\lesssim \mu^{\frac{3}{2}} \| S_{1} \langle \Omega \rangle \psi_{1}\|_{L^{\infty}_{t}L^{2}_{x}} \| S_{1} \langle \Omega\rangle \psi_{2}\|_{L^{\infty}_{t}L^{2}_{x}}\lesssim \mu^{\frac{3}{2}}
\| S_{1} \psi_{2}\|_{\widetilde{F}_{\Omega}} \| S_{1}\psi_{1}\|_{\widetilde{F}_{\Omega}}.
\end{equation}

To sum up, from (\ref{equation22062}), (\ref{equation22063}), (\ref{equation22060}) and (\ref{equation22061}), the desired estimate (\ref{equation22002}) holds when $a\cdot b > 0$. When $a\cdot b < 0$ and $\sigma \leq c$, then both inputs can't be sufficiently close to their corresponding light cones at the same time. The following holds when $a \cdot b < 0$, $\sigma \leq c$:

\begin{equation}\label{equation50000}
P_{\mu} S_{\sigma, \sigma}Q_{a,b} (S_{1} \psi_{1}, S_{1} \psi_{2})  = P_{\mu} S_{\sigma, \sigma}Q_{a,b} (S_{1, \bullet \geq c }^{a} \psi_{1}, S_{1} \psi_{2})  + P_{\mu} S_{\sigma, \sigma}Q_{a,b} (S_{1, \bullet < c}^{a} \psi_{1}, S_{1, \bullet \geq c}^{b} \psi_{2}) .
\end{equation}
As

\[
\sum_{\mu< \sigma \leq c}\| W(  P_{\mu}  S_{\sigma, \sigma} V Q_{a,b} (S_{1, \bullet \geq c }^{a} \psi_{1}, S_{1} \psi_{2}) )\|_{F^{1/2}_{\Omega}} \lesssim \sum_{\mu< \sigma \leq c} \sigma^{-\frac{1}{2}} \mu^{\frac{3}{4}-} \sum_{d\geq c} \| S^{a}_{1,d}\langle \Omega\rangle \psi_{1}\|_{L^{2}_{t}L^{2}_{x}} \| S_{1}\psi_{2}\|_{L^{2}_{t}L^{4+}_{x}}
\]

\begin{equation}\label{equation50001}
+ \sum_{\mu< \sigma \leq c} \sigma^{-\frac{1}{2}} \mu^{\frac{3}{p}} \sum_{d\geq c} d^{-1} \| \sup_{\omega}\|S^{\omega, a}_{1,d} \psi_{1}\|_{L^{p}_{x}}\|_{L^{1}_{t}}  \| S_{1}\langle \Omega \rangle \psi_{2}\|_{L^{\infty}_{t}L^{2}_{x}} \lesssim \mu^{\frac{3}{p}-\frac{1}{2}} \| S_{1}\psi_{1}\|_{\widetilde{F}_{\Omega}}\| S_{1} \psi_{2}\|_{\widetilde{F}_{\Omega}},
\end{equation}

\[
\sum_{\mu< \sigma \leq c}\| W(  P_{\mu}  S_{\sigma, \sigma} V Q_{a,b} (S_{1, \bullet < c }^{a} \psi_{1}, S_{1,\bullet \geq c}^{b} \psi_{2}) )\|_{F^{1/2}_{\Omega}} \lesssim \]

\[ \sum_{\mu< \sigma \leq c} \sigma^{-\frac{1}{2}} \mu^{\frac{3}{4}-}\sum_{d\geq c} \| S^{b}_{1,d}\langle \Omega\rangle \psi_{2}\|_{L^{2}_{t}L^{2}_{x}} \| S_{1, \bullet < c}^{a}\psi_{1}\|_{L^{2}_{t}L^{4+}_{x}}+ \sum_{\mu< \sigma \leq c} \sigma^{-\frac{1}{2}} \mu^{\frac{3}{p}} \sum_{d\geq c}d^{-1}  \]

\begin{equation}\label{equation50002}
\| \sup_{\omega}\|S^{\omega, b}_{1,d} \psi_{2}\|_{L^{p}_{x}}\|_{L^{1}_{t}}  \| S_{1,\bullet < c}^{a}\langle \Omega \rangle \psi_{1}\|_{L^{\infty}_{t}L^{2}_{x}} 
\lesssim \mu^{\frac{3}{p}-\frac{1}{2}} \| S_{1}\psi_{1}\|_{\widetilde{F}_{\Omega}}\| S_{1} \psi_{2}\|_{\widetilde{F}_{\Omega}},
\end{equation}

\[
\sum_{c< \sigma \lesssim 1}\| W(  P_{\mu}  S_{\sigma, \sigma} V Q_{a,b} (S_{1} \psi_{1}, S_{1} \psi_{2}) )\|_{F^{1/2}_{\Omega}} \lesssim \mu^{3/2} 
\sum_{c< \sigma \lesssim 1}\sigma^{-\frac{3}{2}}\| S_{1}\langle \Omega\rangle\psi_{1}\|_{L^{\infty}_{t}L^{2}_{x}} \| S_{1}\langle \Omega\rangle\psi_{2}\|_{L^{\infty}_{t}L^{2}_{x}}
\]
\begin{equation}\label{equation50003}
\lesssim \mu^{\frac{3}{2}} \| S_{1}\psi_{1}\|_{\widetilde{F}_{\Omega}}\| S_{1} \psi_{2}\|_{\widetilde{F}_{\Omega}}.
\end{equation}
Hence from (\ref{equation50000}), (\ref{equation50001}), (\ref{equation50002}) and (\ref{equation50003}), we see that desired estimate (\ref{equation22002}) also holds when $a\cdot b < 0$.
\end{proof}

\section{DKG system with nonzero  mass terms}\label{nonzeromass}

\par Our goal in this section is to show that remained nonzero mass cases can be handled very similarly to the
 method we used in section \ref{proof}.
 The main difference is that the characteristic hypersurfaces will change from light cone to hyperboloid.
 The whole argument in section \ref{proof} is very robust, there are only few lemmas that indeed depends on the information about
characteristic hypersurface. 

\par The first lemma  might be affected is the Strichartz estimate for the frequency localized data which has extra regularity in angular variables, while due to the higher linear decay rate of Klein-Gordon part when compares
to the linear wave equation, 
the range we used (mostly $L^{2}_{t}L^{4+}_{x}$ space) in the wave type still valid in the Klein Gordon type, thus we are safe to use the improved Strichartz estimate.
\par The next lemma might be affected is the angular decomposition lemma we introduced in section \ref{Bilinear} and the fact that we could gain $\mu$ (when $\lambda =1$) in the High $\times$ High interaction case for the bilinear operator $Q_{a,b}$ and in the Low $\times$ High interaction for the bilinear operator $\tilde{Q}_{a,b}$.

Let's first consider the wide angle decomposition lemma in the nonzero mass setting. Suppose that the two inputs has space-time frequencies $(\tau_{1}, \xi_{1})$ and $(\tau_{2}, \xi_{2})$ and then the output frequency will be $(\tau_{1}\pm \tau_{2}, \xi_{1}\pm \xi_{2})$, here the specific sign of  $\pm$ depends on the type of bilinear operator. Assume that the maximum and the medium sizes of two inputs and output are at level $\lambda$ and the minimum size is at level $\mu$. Moreover, the modulations of both inputs and output are less than $d$ and $d\leq \mu$. For  $(\tilde{\tau}_{i}, \tilde{\xi}_{i}) \in \{ (\tau_{1}, \xi_{1}), (\tau_{2}, \xi_{2}), (\tau_{1}\pm \tau_{2}, \xi_{1}\pm \xi_{2}) \}$, $i\in \{1, 2\}$, $a, b \in\{M, m\}$ and assume that $|(\tilde{\tau}_{1}, \tilde{\xi}_{1})| \sim \lambda, |(\tilde{\tau}_{2}, \tilde{\xi}_{2})| \sim \mu \ll \lambda$ and $1 \ll \lambda$.  Then the general formula to do the wide angle decomposition is the following:

\begin{equation}\label{equation22080}
\Big| \big|\sqrt{|\tilde{\xi}_{1}|^{2} + a^{2}} \pm_{1} \sqrt{|\tilde{\xi}_{2}|^{2} + b^{2}}\big| - \sqrt{|\tilde{\xi}_{1}\pm_{2} \tilde{\xi}_{2}|^{2} + (2M+m-a-b)^{2}} \Big| \leq d ,
\end{equation}
after calculation, we could derive the left hand side of (\ref{equation22080}) has the following size:
\[
\lambda^{-1}\Big| 2(\sqrt{|\tilde{\xi}_{1}|^{2} + a^{2}}\sqrt{|\tilde{\xi}_{2}|^{2} + b^{2}} - |\tilde{\xi}_{1}| |\tilde{\xi}_{2}|) \pm_{3} [(2M+m-a-b)^{2}- a^{2}-b^{2}]
\]
\[ + 2 |\tilde{\xi}_{1}||\tilde{\xi}_{2}|(1-\pm_{4}\cos(\angle(\tilde{\xi}_{1}, \tilde{\xi}_{2}))) \Big|,
\]
from Cauchy-Schwarz inequality, we know that
\begin{equation}
\sqrt{|\tilde{\xi}_{1}|^{2} + a^{2}}\sqrt{|\tilde{\xi}_{2}|^{2} + b^{2}} - |\tilde{\xi}_{1}| |\tilde{\xi}_{2}| \geq ab,
\end{equation}
hence we can verify that whatever possible choice of $a,b$, the following is always true when $2M \geq m$,
\begin{equation}\label{equation22070}
 2(\sqrt{|\tilde{\xi}_{1}|^{2} + a^{2}}\sqrt{|\tilde{\xi}_{2}|^{2} + b^{2}} - |\tilde{\xi}_{1}| |\tilde{\xi}_{2}|) \pm_{3} [(2M+m-a-b)^{2}- a^{2}-b^{2}] \geq 0
\end{equation}
and it is sharp. The equality holds when $a=b=M$, $\pm_{3}=-$, $2M=m$ or $ab= M m$, $\pm_{3}=+$ $2M=m$, hence the threshold $2M=m$ is also sharp. When $m < 2M$, the left hand side of (\ref{equation22070}) could be negative (e.g when $|\tilde{\xi}_{1}| = |\tilde{\xi}_{2}|$), 
hence the angle between $\tilde{\xi}_{1}$ and $\tilde{\xi}_{2}$ (or $-\tilde{\xi}_{2}$) would be  large and doesn't change too much with respect to the modulation ``$d$'', the scenario mentioned in subsection \ref{maindiff} can also happen. While if $2M \geq m$, (\ref{equation22070}) holds, then angle between $\tilde{\xi}_{1}$ and $\tilde{\xi}_{2}$ (or $-\tilde{\xi}_{2}$) is indeed smaller than $(d/\mu)^{1/2}$, hence the wide angle decomposition formulas also hold when $2M \geq m$. 

\par While for the fact that we can gain $\mu$ in the High $\times$ High interaction case for  the bilinear operator $Q_{a,b}$ and gain $\mu$ in the Low $\times$ High interaction for the bilinear operator $\tilde{Q}_{a, b}$, actually it's valid for all possible $M,m$. We don't need the bound $2M  \geq m$ for this fact. As in the Klein-Gordon case, the problem become inhomogeneous, without scaling, now goal is to gain $\mu/\lambda$.
Suppose that low space-time frequency has size $\mu$, while the two high space-time frequency has size $\lambda$, $\mu
\ll \lambda$ and $1\ll \lambda$. Also the modulations of two high space-time frequency is small (i.e less than a sufficient small constant times $\lambda$), then we have
\[
|\xi_{1}| + \sqrt{|\xi_{1}|^{2} + a^{2}} \sim |\xi_{2}| + \sqrt{|\xi_{2}|^{2} + b^{2}} \sim \lambda, |\xi_{1} - \xi_{2}| \lesssim \mu, \quad a, b,c \in \{m,M\},
\]
\begin{equation}
|\xi_{1}| \sim |\xi_{2}| \sim \lambda, |\xi_{1}-\xi_{2} | \lesssim \mu \Longrightarrow \angle(\xi_{1}, \xi_{2}) \lesssim {\mu}/{\lambda}.
\end{equation}

\begin{remark}
Notice that in above argument to gain $(d/\mu)^{1/2}$ and $\mu/\lambda$, we only considered the high frequency case, i.e $1\ll \lambda$. For the low frequency case, or more precisely  when above argument to gain $(d/\mu)^{1/2}$ and $\mu$ fails, we have $|\xi_{i}| \ll 1$,  $| \tau_{i} \pm |\sqrt{|\xi_{i}|^{2}+ a^{2}}| | \ll 1$ for $i\in\{1,2\}$ and $|\tau_{1} + \tau_{2}| + |\xi_{2}+\xi_{1}|\sim \mu \ll 1$ .
In this scenario, as $m >0$, then  $(-\Box + m^{2}) ^{-1}$ is like a constant in the frequency space and the modulation of output has size $1$, hence we can put the output in $\langle \Omega\rangle^{-1} X_{\mu}^{2,1}\cap Z_{\Omega, \mu}$, and the  Strichartz estimates are sufficient to estimate this case.
\end{remark}

\par Now we know that all lemmas used in section \ref{proof}  are still valid when $2M \geq m >0$, however 
there is still one point need to be mentioned, which is that the Dirac parts of the formulation (\ref{DKG5}) are 
not the standard Klein-Gordon type. Recall the following formulation of nonzero mass DKG system:
\begin{equation}\label{massDKG}
\left\{\begin{array}{lr}
(-i\partial_{t}  + |D| ) \psi_{+} = -M \beta \psi_{-} + \Pi_{+}(D)(\phi \beta \psi), & \\
(-i \partial_{t} - |D|) \psi_{-} = -M \beta \psi_{+} + \Pi_{-}(D) (\phi \beta \psi), & \\
(-\square + m^{2}) \phi = \langle \beta \psi, \psi \rangle. & \\
\end{array}\right.
\end{equation}
By applying  $(-i\partial_{t} \mp |D|)$ on the both hand side of the equations satisfied by $\psi_{\pm}$, and then after a simple 
substitution, we can see that $\psi_{\pm}$ satisfies standard Klein-Gordon type equations. We can derive the following Duhamel formula for the DKG system
\begin{equation}
\phi = \widetilde{\phi}_{0}  + V_{m}\big(\langle \beta \psi, \psi \rangle \big), 
\end{equation}
\[
\psi_{\pm} =
 \widetilde{\psi}_{\pm,0} - V_{M} \big( (-i\partial_{t} \mp |D|) \Pi_{\pm}( \phi\beta \psi )\big) + V_M \big( M\beta  ( \Pi_{\mp}(\phi \beta \psi)) \big)
\]
\begin{equation}\label{equation90000}
=  \widetilde{\psi}_{\pm,0}  -V_{M} (\Pi_{\pm}(-i\p_{t}\mp |D|-M \beta)(\phi \beta\psi)  ).
\end{equation}
 here $\widetilde{\phi}_{0}$ denotes the linear Klein-Gordon solution with initial data $(\phi_{0}, \phi_{1})$, $\widetilde{\psi}_{\pm,0}$ denotes the
 homogeneous linear Klein-Gordon solution associated with $\psi_\pm$ and 
with initial data $(\psi_{\pm,0},\partial_t \psi_{\pm}(0))$, where $\partial_t \psi_\pm (0) $ is derived from (\ref{massDKG}) by evaluating at the initial time $0$. Operator $V_{M}(\cdot)$ denotes the parametrix for the inhomogeneous Klein-Gordon equations, i.e $V_{M}(f)$ denotes the solution of $ (-\square + M^2) u = f$ with zero
initial data. 
\vspace{1\baselineskip}

Recall that in the section \ref{proof}, we need to utilize the fact that for the $\psi_{+}$ (resp. $\psi_{-}$), the  characteristic hypersurface is the lower half cone (resp. upper half cone) instead of the entire cone to do the bilinear decomposition, as otherwise there are exist some cases that the null structure won't be helpful (note the fact that the null structure depends only on the type of inputs while the angle between Dirac input and the output depends on the location of two space-time frequencies). So, for the Klein-Gordon case, we need to pay special attention to the behavior of $\psi_{+} $ (resp. $\psi_{-}$) near the upper half hyperboloid (resp. the lower half hyperboloid). The function spaces we will working on are very similar to the function space we used in the massless case, the major change is that the modulation used for $\psi_{+}$ (resp $\psi_{-}$) will be the distance to the entire hyperboloid instead of the lower half hyperboloid (resp the upper lower half hyperboloid).

 From (\ref{equation90000}), we could see that the symbol of parametrix associated with $\psi_{\pm}$ is the following:
 
\[
L_{\pm}(\tau, \xi):= \frac{\tau\mp |\xi| - M\beta}{|\tau|^{2}-|\xi|^{2} - M^{2}} = \frac{\tau\mp |\xi| - M\beta}{(\tau + \sqrt{|\xi|^{2}+M^{2}}) (\tau- \sqrt{|\xi|^{2} + M^{2}})}.
\]

When the two input frequencies $(\tau_{i},\xi_{i})$, $i\in\{1,2\}$  and the output  frequency  $(\tau, \xi):= (\tau_{1}\pm \tau_{2}, \xi_{1}\pm \xi_{2})$are all near the hyperboloid, and suppose that the modulations (distance to the entire hyperboloid) are all less than $d$, then we have
\begin{equation}
\big| | \sqrt{|\xi_{1}|^{2}+ M^{2}}  - \sqrt{|\xi_{2}|^{2}+ M^{2}}|  - \sqrt{|\xi|^{2}+m^{2}}\big| \lesssim d,
\end{equation}
similar to the analysis in the (\ref{equation22080}) case, when $2M > m$, we can show that
\begin{equation}\label{equation90003}
d\gtrsim_{m,M} \frac{1}{\sqrt{|\xi|^{2} + m^{2}}}.
\end{equation}
Hence, within the possible range of modulation, when the space-time frequency of output $\psi_{+}$ is very close to the upper half hyperboloid, then we have 
\begin{equation}\label{equation90002}
|L_{+}(\tau,\xi)c^{-}_{d}(\tau, \xi)|  \lesssim \frac{d_{-}+M}{d_{-}\sqrt{|
\xi|^{2}+M^{2}}}, \quad |d_{-}|:=|\tau -\sqrt{|\xi|^{2}+M^{2}}| \ll \sqrt{|\xi|^{2}+M^{2}}.
\end{equation}
One can verify that the upper bound of (\ref{equation90002}) is sufficient to get desired estimate without using null structure when the space-time frequency of output is close to the upper hyperboloid, as the modulation has the lower bound (see(\ref{equation90003})), we don't need to worry about the summation problem with respect to $d$. When the output frequency is near the lower half hyperboloid, we can utilizing the null structure and the symbol of parametrix $|L_{+}(\tau,\xi)c^{+}_{d}(\tau, \xi)|  \lesssim d_{+}^{-1}$ is of the same type as the one appeared in the massless case, hence, with minor
 modifications we can also get desired estimate when the output space-time frequency is very close to the lower half hyperboloid. $\psi_{-}$ can be estimated in the same way. Hence the main theorem also holds when the mass terms presents in the DKG system when $2M  > m.$


\begin{thebibliography}{99}
\bibitem{Bachelot} A. Bachelot. Probl\'eme de Cauchy pour des systemes hyperboliques semi-lineaires, \textit{Ann. Inst. H. Poincare Anal. Non Lineaire}, {1}(1984), pp 453-478.

\bibitem{Bournaveas} N. Bournaveas. Local existence of energy class solutions for the Dirac-Klein-Gordon equations, \textit{Comm. Partial Differential Equations}, {24} (1999), pp 1167-1193.
\bibitem{Bournaveas1} N. Bournaveas. A new proof of global existence for the Dirac-Klein-Gordon equations in on space dimension, \textit{J. Funct. Anal} {173}(2000) pp 203-213
.

\bibitem{Bournaveas2} N. Bournaveas. Low regularity solutions of the Dirac-Klein-Gordon equations in two space dimensions, \textit{Comm. Partial Differential Equations} {26}(2001), pp 1345-1366.
\bibitem{Chadam1} J.M Chadam. Global solutions of the Cauchy problem for the (classcial) coupled Maxwell- Dirac equations in one space dimensions, \textit{J. Funct. Anal} {13} (1973), pp 173-184.

\bibitem{Holmer} J. Colliander, J. Holmer, N. Tzirakis. Low regularity global well-posedness for the Zakharov and Klein-Gordon-Schr\"odinger system,\textit{Trans. Amer. Math. Soc} {360}(2008), pp 4619-4638.
\bibitem{D'Ancona1} Piero D'Ancona, Damiano Foschi, and Sigmund Selberg. Null structure and almost optimal local regularity for the Dirac-Klein-Gordon System, \textit{J.Eur. Math. Soc} (2007), pp 877-899.

\bibitem{D'Ancona2} P. D'Ancona, D. Foschi, S. Selberg. Local well-posedness below the charge norm for the Dirac-Klein-Gordon system in two space dimensions, \textit{J. Hyper. Differential Equations} (2007), pp 295-330.

\bibitem{D'Ancona3} P. D'Ancona, S. Selberg. Global well-posedness of the Maxwell-Dirac system in two space dimensions, \textit{J. Func. Anal} {260}
(2011), pp 2300-2365.


\bibitem{Fang} Y-F. Fang. On the Dirac-Klein-Gordon equations in one space dimension, \textit{Differential Integral Equations} {17} (2004), pp 1321-1346.

\bibitem{Fang2} Y-F. Fang, M. Grillakis. On the Dirac-Klein-Gordon equations in three space dimensions, \textit{Comm. Partial Differential Equations} {30}(2005), pp 783-812.

\bibitem{Foschi} D. Foschi, S. Klainerman. Bilinear Space-time estimates for homogeneous wave equations, \textit{Ann. Sci. $\acute{E}$cole Norm. Sup} {23}(2000), pp 211-274.

\bibitem{Pecher2} A. Gr\"unrock, H. Pecher. Global solutions for the Dirac-Klein-Gordon system in two space dimensions, \textit{Comm. Partial Differential Equations} {1}(2010), pp 89-112.

\bibitem{Klainerman2} S. Klainerman, M. Machedon. Space-time estimates for null forms and the local existence theorem, \textit{Comm. Pure Appl. Math} {46} (1993), pp 1221-1268.
\bibitem{Klainerman} S. Klainerman, M. Machedon. On the regularity properties of the wave equation. \textit{Physics on Manifolds} (Paris 1992), Math. Phy. Stud 15, {1994}, pp 177-191.

\bibitem{Klainerman1} S. Klainerman, D. Tataru. On the optimal local regularity for Yang-Mills equations in $\mathbb{R}^{4+1}$, \textit{J. Amer. Math. Soc} {12}(1999), pp 93-116.


\bibitem{Klainerman3} S. Klainerman, S. Selberg. Bilinear estimates and applications to nonlinear wave equations, \textit{Commum. Contemp. Math} {4}(2002) pp 223-295.

\bibitem{Klainerman4} S. Klainerman, M. Machedon. Remark on Strichartz type inequalities, \textit{Int. Math. Res. Not} (2006) no. {5}, pp 201-220.
\bibitem{Lindblad} H. Lindblad. Counter examples to local existence for semi-linear wave equations, \textit{Amer. J. Math} {118} (1996), pp 1-16.

\bibitem{Machedon} M. Machedon, J. Sterbenz. Almost optimal local well-posedness for the (3+1)-dimensional Maxwell-Klein-Gordon equations, \textit{J. Amer. Math. Soc} {17}(2004) pp 297-359.



\bibitem{Pecher1} H. Pecher. Low regularity well-posedness for the one-dimensional Dirac-Klein-Gordon system, \textit{Electron. J. Differential Equations} (2006), no 150, 13 pp.

\bibitem{Peskin} M. Peskin, D. Schroeder. An introduction to quantum field theory (Addison-Wesley, New York,1995).

\bibitem{Ponce} G. Ponce, T. Sideris. Local regularity of nonlinear wave equations
in three space dimensions, \textit{Comm. Partial Differential Equations}, {18},
 169--177 (1993).
\bibitem{Sterbenz1} J. Sterbenz. Global regularity for general non-linear wave equations {I}: $(6+1)$ and higher dimensions, \textit{Comm. Partial Differential Equations} {29}(2005), pp 1505-1531.

\bibitem{Sterbenz3} J. Sterbenz. Angular regularity and Strichartz estimates for the wave equation, \textit{Int. Math. Res. Notices}(2005), pp 187-231.
\bibitem{Sterbenz2} J. Sterbenz.  Global regularity and scattering for general non-linear wave equations II. $(4+1)$ dimensional Yang-Mills equations in the Lorentz gauge, \textit{Amer. J. Math} {129}(2007), pp 611-664.

\bibitem{Tao} T. Tao. Low regularity semi-linear wave equations, \textit{Comm. Partial Differential Equations}
{24}(1999), pp 599-629.

\bibitem{Tao2} T. Tao. Global regularity of wave maps, I: small
critical Sobolev norm in high dimension, \textit{Int Math Res Notices}, (2001) 2001 (6), 299-328.

\bibitem{Tao1} T. Tao. Global regularity of wave maps II. small
energy in two dimensions, \textit{Comm. Math. Phys}, 224 (2001), 443-544.

\bibitem{Tataru1} D. Tataru. On the equation $\square u = |\nabla u |^{2}$ in $5+1$ dimensions, \textit{Math. Res. Lett}(1999), no 5-6, pp 469-485.

\bibitem{Tataru2} D. Tataru. On global existence and scattering for the wave maps equation, \textit{Amer. J . Math}(2001) Vol 123, pp 37-77.
\end{thebibliography}
\end{document}